\documentclass[12pt, a4paper]{amsart}
\usepackage{amsmath,mathtools}
\usepackage{amsthm,graphics,tabularx,amssymb,shapepar}
\usepackage{fullpage}
\usepackage{enumerate,enumitem}
\usepackage{tikz-cd}
\usepackage{amscd}
\usepackage{hyperref}

\makeatletter
\newcommand*{\rom}[1]{\expandafter\@slowromancap\romannumeral #1@}
\makeatother

\newcommand{\BA}{{\mathbb {A}}}

\newcommand{\BC}{{\mathbb {C}}}

\newcommand{\BK}{{\mathbb {K}}}

\newcommand{\BR}{{\mathbb {R}}}

\newcommand{\BZ}{{\mathbb {Z}}}

\newcommand{\CC}{{\mathcal {C}}}

\newcommand{\CE}{{\mathcal {E}}}

\newcommand{\CN}{{\mathcal {N}}}
\newcommand{\CO}{{\mathcal {O}}}

\newcommand{\CS}{{\mathcal {S}}}
\newcommand{\CT}{{\mathcal {T}}}

\newcommand{\RM}{{\mathrm {M}}}
\newcommand{\RO}{{\mathrm {O}}}

\newcommand{\RU}{{\mathrm {U}}}

\newcommand{\Ad}{{\mathrm{Ad}}}

\newcommand{\GL}{{\mathrm{GL}}}

\newcommand{\GSp}{{\mathrm{GSp}}}
\newcommand{\Hom}{{\mathrm{Hom}}}

\renewcommand{\Im}{{\mathrm{Im}}}
\newcommand{\Ind}{{\mathrm{Ind}}}
\newcommand{\ind}{{\mathrm{ind}}}

\newcommand{\Irr}{{\mathrm{Irr}}}

\newcommand{\Ker}{{\mathrm{Ker}}}

\newcommand{\Lie}{{\mathrm{Lie}}}

\renewcommand{\Re}{{\mathrm{Re}}}

\newcommand{\SL}{{\mathrm{SL}}}

\newcommand{\SO}{{\mathrm{SO}}}

\newcommand{\Sym}{{\mathrm{Sym}}}

\newcommand{\Span}{{\mathrm{Span}}}

\newcommand{\Ima}{\operatorname{Im}}

\newcommand{\Tr}{\operatorname{Tr}}

\newcommand{\sgn}{\operatorname{sgn}}

\newcommand{\oH}{\operatorname{H}}

\newcommand{\g}{\mathfrak g}
\newcommand{\gC}{{\mathfrak g}_{\C}}
\renewcommand{\k}{\mathfrak k}
\newcommand{\h}{\mathfrak h}
\newcommand{\p}{\mathfrak p}

\newcommand{\n}{\mathfrak n}
\renewcommand{\u}{\mathfrak u}
\renewcommand{\v}{\mathfrak v}

\renewcommand{\l}{\mathfrak l}
\renewcommand{\t}{\mathfrak t}
\newcommand{\s}{\mathfrak s}

\newcommand{\gl}{\mathfrak g \mathfrak l}

\newcommand{\fS}{\mathfrak S}

\newcommand{\C}{\mathbb{C}}

\newcommand{\be}{\begin {equation}}
\newcommand{\ee}{\end {equation}}
\newcommand{\bee}{\begin {equation*}}
\newcommand{\eee}{\end {equation*}}

\newcommand{\Fre}{{Fr\'{e}chet \,}} 
\newcommand{\fre}{{Fr\'{e}chet }}
\newcommand{\kun}{{K\"{u}nneth }}
\newcommand{\smod}{\CS\mathrm{mod}}

\renewcommand{\mid}{\,:\,}

\newtheorem{thm}{Theorem}[section]
\newtheorem{cort}[thm]{Corollary}
\newtheorem{lemt}[thm]{Lemma}
\newtheorem{prpt}[thm]{Proposition}

\newtheorem{dfnt}[thm]{Definition}

\theoremstyle{remark}
\newtheorem*{example}{Example}

\theoremstyle{remark}
\newtheorem*{rremark}{Remark}

\numberwithin{equation}{section}

\begin{document}
\title{On the existence of twisted Shalika periods: the Archimedean case}
\date{\today}
\author[Z. Geng]{Zhibin Geng}
\address{Academy of Mathematics and Systems Science, Chinese Academy of Science, No. 55 Zhongguancun East Road, Beijing, 100190, China.}
\email{gengzhibin@amss.ac.cn}

\keywords{Twisted Shalika periods, L-parameters of symplectic type, Schwartz homology}
\subjclass[2020]{22E50, 11F70, 20G20}

\begin{abstract}
    Let $\BK$ be an archimedean local field. We investigate the existence of the twisted Shalika functionals on irreducible admissible smooth representations of $\GL_{2n}(\BK)$ in terms of their L-parameters. 
    As part of our proof, we establish a Hochschild-Serre spectral sequence for nilpotent normal subgroups and a \kun formula in the framework of Schwartz homology.
    We also prove the analogous result for twisted linear periods using theta correspondence. 
    The existence of twisted Shalika functionals on representations of $\GL_{2n}^{+}(\BR)$ is also studied, which is of independent interest.
\end{abstract}

\maketitle
\tableofcontents

\section{Introduction}

\subsection{Main results}
Let $\BK$ be a local field of characteristic zero. The Shalika subgroup of the general linear group $\GL_{2n}(\BK)\ (n\geq1)$ is defined to be
\[
S_{2n}(\BK) := \left\{ 
\begin{bmatrix} 
g & 0 \\ 0 & g 
\end{bmatrix} \cdot 
\begin{bmatrix} 
1_n & b \\ 0 & 1_n 
\end{bmatrix} \mid g \in \GL_n(\BK),\ b \in \RM_n(\BK) \right\},
\]
where $\RM_n$ indicates the algebra of $n \times n$ matrices.
Let $\eta$ be a character of $\BK^{\times}$, and let $\psi$ be a non-trivial unitary character of $\BK$. Define a character $\xi_{\eta,\psi}$ on $S_{2n}(\BK)$ by
\[
\xi_{\eta,\psi} \left( \begin{bmatrix} g & 0 \\ 0 & g \end{bmatrix} \cdot \begin{bmatrix} 1_n & b \\ 0 & 1_n \end{bmatrix}\right) := \eta(\det (g))\psi(\Tr(b)).
\]
\par
For an admissible smooth representation $\pi$ of $\GL_{2n}(\BK)$, we say that $\pi$ has a non-zero \textbf{$(\eta,\psi)$-twisted Shalika period} if $$\Hom_{S_{2n}(\BK)}(\pi,\xi_{\eta,\psi}) \neq \{0\}.$$ 
When $\eta$ is the trivial character, we shall refer to it simply as the \textbf{Shalika period}.
Here and henceforth, for the archimedean local field $\BK$, by an admissible smooth representation of $\GL_n(\BK)$, we mean a Casselman-Wallach representation of it. Recall that a representation of a real reductive group is called a Casselman-Wallach representation if it is Fr\'{e}chet, smooth, admissible, and of moderate growth. The readers may consult \cite[Chapter 11]{Wallach1992Real} or \cite{Bernstein2014Smooth} for details. It is proved in \cite{chen2020uniqueness} that the space $\Hom_{S_{2n}(\BK)}(\pi,\xi_{\eta,\psi})$ is at most one-dimensional if $\pi$ is irreducible. 
\par
A natural question is when the space $\Hom_{S_{2n}(\BK)}(\pi,\xi_{\eta,\psi})$ is non-zero, which is commonly referred to as the distinction problem.
In the framework of the relative Langlands program, this question is related to the BZSV quadruple $$(\GL_{2n} \times \GL_1, \GL_n, 0, \iota: \GL_n \times \SL_2 \to \GL_{2n} \times \GL_1),$$
where 
$$
\begin{array}{rccc}
\iota :& \GL_n \times \SL_2 &\to &\GL_{2n} \times \GL_1 \\
&(A,B)& \mapsto &(A \otimes B, \det A),
\end{array}
$$
while the dual BZSV quadruple of it is 
$$(\GL_{2n} \times \GL_1, \GSp_{2n}, 0, (i,\lambda): \GSp_{2n} \to \GL_{2n} \times \GL_1 ),$$
where $i: \GSp_{2n} \to \GL_{2n}$ is the inclusion,
$\lambda: \GSp_{2n} \to \GL_1$ is the similitude character of $\GSp_{2n}$. The local aspect of this framework suggests that the dual quadruple is related to the distinction problem, which will be confirmed by our main result.
\par
From now on, we will assume that $\BK$ is an archimedean local field. Let $W_{\BK}$ be the Weil group of $\BK$.
\begin{dfnt}
    For an L-parameter $\phi: W_{\BK} \to \GL_{2n}(\mathbb{C})$, its \textbf{$\eta$-extension} is defined to be 
    $$
    \begin{array}{rccrclc}
    \phi^{(\eta)} :& W_{\BK} &\to &\GL_{2n}(\mathbb{C})&  \times & \GL_1(\mathbb{C})& \\
    &x& \mapsto & ( \ \phi(x) \ \ &  , &  \eta(r(x)) ) &.
    \end{array}
    $$
    where $r : W_{\BK} \to \BK^{\times}$ is the reciprocity map (will be recalled in Subsection \ref{Subsec: LLC}).
    \par
    An L-parameter $\phi$ is said to be \textbf{of $\eta$-symplectic type}, if its $\eta$-extension $\phi^{(\eta)}$ factors through the above map $(i,\lambda) : \GSp_{2n}(\BC) \to \GL_{2n}(\mathbb{C})  \times   \GL_1(\mathbb{C})$.
\end{dfnt}

The main result of this paper is the following. 

\begin{thm}\label{Main thm}
    Let $\pi$ be an irreducible Casselman-Wallach representation of $\GL_{2n}(\BK)$. 
    \begin{enumerate}[label=(\Alph*), ref=\ref{Main thm}\Alph*]
        \item $(\S \ref{Subsec: thmA})$
        If $\pi$ has a non-zero $(\eta,\psi)$-twisted Shalika period, then its L-parameter is of $\eta$-symplectic type. \label{Main thmA}
        \item $(\S \ref{Sec: thmB})$
        Assume that $\pi$ is generic. If the L-parameter of $\pi$ is of $\eta$-symplectic type, then $\pi$ admits a $(\eta,\psi)$-twisted Shalika period. \label{Main thmB}
    \end{enumerate}
\end{thm}

To summarize, we have the following equivalence between the existence of the twisted Shalika period and the type of L-parameter.

\begin{thm}\label{thmC}
    For an irreducible generic Casselman-Wallach representation $\pi$ of $\GL_{2n}(\BK)$, it has a non-zero $(\eta,\psi)$-twisted Shalika period if and only if its L-parameter is of $\eta$-symplectic type.
\end{thm}

\begin{rremark}
    The generic condition cannot be removed without additional assumptions. For example, the trivial representation of $\GL_2(\BR)$ doesn't have Shalika periods, while its L-parameter is of symplectic type.
\end{rremark}

When $\pi$ is an irreducible essentially tempered cohomological representation, Theorem \ref{thmC} has been proved in \cite{chen2020archimedean} for $\BK = \BR$ and \cite{lin2020archimedean} for $\BK = \BC$. When $\pi$ is a principal series representation, \cite{Jiang2024twisted} also explores the existence of the twisted Shalika periods under certain conditions.
When $\eta$ is a trivial character, Theorem \ref{thmC} is proved in \cite{matringe2017shalika} for the non-archimedean case and \cite{anandavardhanan2024sign} for the archimedean case. 
\par

All the aforementioned works (except \cite{chen2020archimedean} and \cite{Jiang2024twisted}) study the analogous properties of twisted linear periods, then transfer them to the twisted Shalika periods using the analytic method. By comparison, Our proof provides a direct analysis of the twisted Shalika periods and has no restriction on $\eta$.
\par

The theory of theta correspondence gives a way to relate different kinds of periods. Let $H_{2n}(\BK) := \GL_n(\BK) \times \GL_n(\BK)$, and view it as a subgroup of $\GL_{2n}(\BK)$ via the diagonal embedding. For an admissible smooth representation $\pi$ of $\GL_{2n}(\BK)$, we say that $\pi$ has a non-zero \textbf{$\eta$-twisted linear period} if $$\Hom_{H_{2n}(\BK)}(\pi,(\eta \circ \det) \boxtimes \BC) \neq \{0\}.$$
Using theta correspondence, it is shown in \cite{Gan2019} that Shalika periods are related to linear periods over non-archimedean local fields. The following theorem is an Archimedean analog of it.

\begin{thm}[Theorem \ref{Thm: archimedean theta shalika linear}]\label{thmD}
     Let $\BK$ be an archimedean local field. Let $\pi$ be an irreducible Casselman-Wallach representation of $\GL_{2n}(\BK)$ and $\chi$ be a character of $\GL_{n}(\BK)$. Denote by $\Theta(\pi)$ the big theta lift of a representation $\pi$, one has
    $$
    \Hom_{\GL_n^{\Delta}(\BK)\ltimes U} (\Theta(\pi), \chi \boxtimes \psi) \cong \Hom_{\GL_n(\BK)\times \GL_n(\BK)}(\pi^{\vee}, \chi \boxtimes \BC),
    $$
    where $S_{2n}(\BK) = \GL_n^{\Delta}(\BK)\ltimes U$.
    Moreover, for generic $\pi$ we have
    $$
    \Hom_{\GL_n^{\Delta}(\BK)\ltimes U} (\pi, \chi \boxtimes \psi) \cong \Hom_{\GL_n(\BK)\times \GL_n(\BK)}(\pi, \chi \boxtimes \BC)
    $$
\end{thm}
Combining Theorem \ref{thmC} and Theorem \ref{thmD}, we get the following corollary.
\begin{cort}
    Let $\BK$ be an archimedean local field. Let $\pi$ be an irreducible generic Casselman-Wallach representation of $\GL_{2n}(\BK)$. The following are equivalent: \par
    \begin{itemize}
        \item [(1)]
        The L-parameter of $\pi$ is of $\eta$-symplectic type.
        \item [(2)]
        $\pi$ has a $(\eta,\psi)$-twisted Shalika periods.
        \item [(3)]
        $\pi$ has a $\eta$-twisted linear periods.
    \end{itemize}
\end{cort}

Twisted Shalika periods are important in the study of L-functions of symplectic type. 
In \cite{jiang2024period}, they study the period relations for the critical values of the standard L-functions for an irreducible regular algebraic cuspidal automorphic representations of $\GL_{2n}(\BA)$ of symplectic type. During their proof, a constant $\epsilon_{\pi}$ has been defined (and will be recalled in Section \ref{Sec: restrict to GL+}, \eqref{Def: epsilon pi}), which reflects the existence of the twisted Shalika periods on the representations of $\GL^{+}_{2n}(\BR)$.
The following theorem provides the calculation of $\epsilon_{\pi}$.
\begin{thm}[Theorem \ref{Thm: restrict Shalika support}]\label{thmE}
    The notation is as shown in Section \ref{Sec: restrict to GL+}.
    Let $\pi$ be an irreducible generic Casselman-Wallach representation of $\GL_{2n}(\BR)$. 
    For $a \in \BR^{\times}$, denote 
    $$\psi_a : \BR \to \BC^{\times}, x \mapsto \exp(2\pi ax\sqrt{-1}).$$
    Assume $\pi$ admits a non-zero $(\eta,\psi_a)$-twisted Shalika period and $\pi|_{\GL^{+}_{2n}(\BR)}$ is reducible, then $\pi$ has form
    $$
    D_{k_1, \lambda_1} \dot{\times} \cdots \dot{\times} D_{k_n, \lambda_n},\ k_i \in \BZ_{\geq 0},\ \lambda_i \in \BC.
    $$
    Let
    $
    p := \#\{1 \leq i \leq n | D_{k_i, \lambda_i} \text{ has } (\eta,\psi_a)\text{-twisted Shalika period} \}
    $
    and $q := \frac{n-p}{2}$, which are both integers since $\phi_{\pi}$ is of $\eta$-symplectic type. Then
        $$
        \epsilon_{\pi} = (\sgn\ a)^p \cdot (-1)^{\frac{p(p-1)}{2}+q}.
        $$
\end{thm}

\subsection{Outline of the proof}\label{Subsec: outline}
We outline a general framework for addressing such problems as Theorem \ref{Main thmA} and Theorem \ref{Main thmB}, and specialize it to our case at the end. We shall use the terminology in Schwartz homology developed in \cite{CHEN2021108817}, which will be reviewed in Section \ref{Sec: schwartz homology}. The proof of Theorem \ref{Main thmA} is inspired by  \cite{suzuki2023epsilon}. 
\par
Let $G$ be an almost linear Nash group and let $H$ be a Nash subgroup of $G$. Denote by $\CS\mathrm{mod}_G$ the category of smooth \fre representations of $G$ of moderate growth.
Let $X$ be a $G$-Nash manifold and let $\CE$ be a tempered $G$-vector bundle on $X$. Denote by $\Gamma^{\varsigma}(X, \CE)$ the Schwartz sections of the tempered bundle $\CE$ over $X$, then $\Gamma^{\varsigma}(X, \CE) \in \smod_G$. Take $x \in X$, denote by $\CE_x$ the fiber of $\CE$ at $x$ and $H_x$ the stabilizer in $H$ at $x$. Let $\chi$ be a character of $H$. 
Note that $\Hom_H(\Gamma^{\varsigma}(X, \CE), \chi)$ is the continuous linear dual of the zeroth Schwartz homology $\oH^\CS_{0}(H, \Gamma^{\varsigma}(X, \CE) \otimes \chi^{-1})$. 
In the following, we investigate the relationship between $\oH^\CS_{*}(H, \Gamma^{\varsigma}(X, \CE) \otimes \chi^{-1})$ and $\oH^\CS_{*}(H_x, \CE_x \otimes \chi^{-1})$.
\par
Assume that $X$ admits a finite decreasing sequence of open submanifolds 
$$U_1 := X \supsetneq U_2 \supsetneq ... \supsetneq U_{r} \supsetneq U_{r+1} := \varnothing,$$
such that for $i \in \{1, \dots, r \}$, $O_{i} := U_i\backslash U_{i+1}$ is an $H$-orbit in $X$. For each $i \in \{1, \dots, r\}$, we have the following short exact sequence
\begin{equation}
0 \to \Gamma^{\varsigma}(U_{i+1}, \CE) \to \Gamma^{\varsigma}(U_i, \CE) \to \Gamma^{\varsigma}_{O_{i}}(U_i,\CE) \to 0.
\end{equation}
After twisting the character $\chi^{-1}$ of $H$ and taking Schwartz homology with respect to $H$, we get the long exact sequence
\begin{multline}\label{Eq: outline first LES}
\cdots\to \oH^\CS_{j}(H, \Gamma^{\varsigma}(U_{i+1}, \CE) \otimes \chi^{-1}) \to \oH^\CS_{j}(H, \Gamma^{\varsigma}(U_i, \CE) \otimes \chi^{-1}) \to \\
\oH^\CS_{j}(H, \Gamma^{\varsigma}_{O_{i}}(U_i,\CE) \otimes \chi^{-1}) \to \cdots.
\end{multline}
\par
We may encounter two different situations in practice. In the first case (corresponding to the case of Theorem \ref{Main thmA}), we know that $\oH^\CS_{0}(H, \Gamma^{\varsigma}(X, \CE) \otimes \chi^{-1}) \neq 0$, and we want to deduce expected properties on the representation $\CE_x$ of $H_x$. Following from the long exact sequence (\ref{Eq: outline first LES}), we have
$$
\dim \oH^\CS_{j}(H, \Gamma^{\varsigma}(X, \CE) \otimes \chi^{-1}) \leq \sum_{i=1}^{r} \dim \oH^\CS_{j}(H, \Gamma^{\varsigma}_{O_{i}}(U_i,\CE) \otimes \chi^{-1}).
$$
Thus we know there exist an $l \in \{1, \dots, r \}$ such that
$$
\dim \oH^\CS_{0}(H, \Gamma^{\varsigma}_{O_{l}}(U_l,\CE) \otimes \chi^{-1}) \neq 0.
$$
In the second case (corresponding to the case of Theorem \ref{Main thmB}), we know that $\oH^\CS_{0}(H, \Gamma^{\varsigma}(\CO_r, \CE) \otimes \chi^{-1}) \neq 0$ for the open orbit $\CO_r \subset X$, and we want to prove that the map induced by extension by zero is an isomorphism on the zeroth homology, i.e.
$$
\oH^\CS_{0}(H, \Gamma^{\varsigma}(\CO_r, \CE) \otimes \chi^{-1}) \cong \oH^\CS_{0}(H, \Gamma^{\varsigma}(X, \CE) \otimes \chi^{-1}).
$$
Then it suffices to prove that for $i \in \{1, \dots, r-1\},$
$$
\oH^\CS_{0}(H, \Gamma^{\varsigma}_{O_{i}}(U_i,\CE) \otimes \chi^{-1}) =
\oH^\CS_{1}(H, \Gamma^{\varsigma}_{O_{i}}(U_i,\CE) \otimes \chi^{-1}) = 0.
$$
To summarize, in both cases, we reduce the problem to the homological property of $\Gamma^{\varsigma}_{O_{i}}(U_i,\CE)$.
\par
Using Borel's lemma, $\Gamma^{\varsigma}_{O_{i}}(U_i,\CE)$ admits a decreasing filtration $\Gamma^{\CS}_{O_{i}}(U_i,\CE)_k,\ k = 0,1,2,...$, such that
$$
\Gamma^{\CS}_{O_{i}}(U_i,\CE) \cong  \varprojlim_k \, \Gamma_{O_{i}}^{\CS}(U_i, \CE)/\Gamma_{{O_{i}}}^{\CS}(U_i, \CE)_k ,
$$
with the graded pieces
$$
 \Gamma_{{O_{i}}}^{\CS}(U_i, \CE)_k/\Gamma_{{O_{i}}}^{\CS}(U_i, \CE)_{k+1}\cong \Gamma^{\CS}({O_{i}}, \Sym^k(\CN^{\ast}_{O_{i}}(U_i))\otimes \CE), \quad (k\geq 0).
$$
Write $\Gamma_{i,k}^{\CS} := \Gamma_{{O_{i}}}^{\CS}(U_i, \CE)_k,\ k = 0,1,2,...$ for short. We have a short exact sequence
$$
0 \to \Gamma_{{i,k}}^{\CS} / \Gamma_{{i,k+1}}^{\CS} \to \Gamma_{i,0}^{\CS}/\Gamma_{{i,k+1}}^{\CS} \to \Gamma_{i,0}^{\CS}/\Gamma_{{i,k}}^{\CS}\to 0,
$$
and hence a long exact sequence
\begin{multline}\label{Eq: second LES outline}
\cdots \to \oH^\CS_{j}(H, (\Gamma_{{i,k}}^{\CS} / \Gamma_{{i,k+1}}^{\CS}) \otimes \chi^{-1}) \to \oH^\CS_{j}(H, (\Gamma_{i,0}^{\CS}/\Gamma_{{i,k+1}}^{\CS}) \otimes \chi^{-1}) \to \\
\oH^\CS_{j}(H, (\Gamma_{i,0}^{\CS}/\Gamma_{{i,k}}^{\CS}) \otimes \chi^{-1}) \to \cdots.
\end{multline}
\par
Assume that
\begin{equation}\label{Eq: homological finiteness}
\dim \oH^\CS_{j+1}(H, (\Gamma_{i,0}^{\CS}/\Gamma_{{i,k}}^{\CS}) \otimes \chi^{-1}) < \infty, \quad \forall\ k\geq 0. 
\end{equation}
According to Proposition \ref{Prop: homology borel lem}, we have
\[
\oH_{j}^\CS(H, \Gamma^{\CS}_{\CO_{i}}(U_i,\CE) \otimes \chi) \cong \varprojlim_{k}\oH_{j}^\CS(H, (\Gamma_{i,0}^{\CS}/\Gamma_{{i,k}}^{\CS})\otimes \chi).
\]
Using Borel's lemma and the long exact sequence (\ref{Eq: second LES outline}), we may reduce the problem to the calculation of  
$$
\oH^\CS_{j}(H, \Gamma^{\CS}({O_{i}}, \Sym^k(\CN^{\ast}_{O_{i}}(U_i))\otimes \CE) \otimes \chi^{-1}),
$$
which, using Shapiro's lemma, is related to the homological property of $\CE_x$ with respect to $H_x$ for $x \in \CO_i$.
\par
Now we return to our case. Let $X:= P\backslash \GL_{2n}(\BK)$ be a partial flag manifold, $\CE$ the tempered vector bundle on $X$ associated with a standard module, $S:= S_{2n}(\BK)$ the Shalika subgroup, and $\chi:= \xi_{\eta,\psi}$.
\par
To prove Theorem \ref{Main thmA}, we first describe the $S$-orbit decomposition of $X$ in Section \ref{Sec: Orbit decomposition}, and give the definition of matching orbits. Then we prove in Lemma \ref{Lem: orbitwise homological finiteness} that the homology groups occurring in \eqref{Eq: homological finiteness} are finite-dimensional. Finally, using various tools of the Schwartz homology and the calculation of the Jacquet module, we prove that 
\begin{equation*}
    \oH^\CS_{0}(S, \Gamma^{\varsigma}_{O_{i}}(U_i,\CE) \otimes \xi_{\eta,\psi}^{-1})
    \cong
    \left\{ 
        \begin{array}{ll}
        \oH^\CS_{0}(S, \Gamma^{\varsigma}(\CO_i,\CE) \otimes \xi_{\eta,\psi}^{-1}), & \textrm{if $\CO_i$ is a matching orbit};\\
        0, &  \textrm{otherwise}.
        \end{array}
    \right.
\end{equation*}
and deduce Theorem \ref{Main thmA} from this using Proposition \ref{Prop: symplectic L parameter and homology}.
\par
The existence of the twisted Shalika functional on the parabolic induced representation has been studied in \cite[Theorem 2.1]{chen2020archimedean}. Using this, we reduce the proof of Theorem \ref{Main thmB} to the case for $\GL_4(\BR)$. Using the framework described above, we prove that, in this case, for all non-open orbits $\CO_i$ and $\forall j \in \BZ$,
$$
\oH^\CS_{j}(S_4(\BR), \Gamma^{\varsigma}_{O_{i}}(U_i,\CE) \otimes \xi^{-1}_{\eta,\psi}) = 0,
$$
and then deduce the result from the long exact sequence (\ref{Eq: outline first LES}).
\par
The proof of Theorem \ref{thmD} is similar to the non-archimedean case as in \cite{Gan2019}. However, the calculation of the coinvariance space of the Weil representation is more involved in the archimedean case. We use \cite[Lemma 6.2.2]{aizenbud2015twisted} to show that the result is just the same as in the non-archimedean case.
\par
For proving Theorem \ref{thmE}, we study the restriction of certain parabolic induced representations to $\GL^{+}_{2n}(\BR)$, then the similar arguments as in Section \ref{Sec: thmB} can be applied.

\subsection{Structure of the paper}
We now describe the contents and the organization of this paper. In Section \ref{Sec: preliminaries}, we give the necessary preliminaries on the local Langlands correspondence for $\GL_n(\BK)$ and introduce basic properties of nuclear \Fre spaces. \par
In Section \ref{Sec: schwartz homology}, we review basic knowledge about Schwartz homology and present the propositions which are needed in our subsequent proofs. Certain homology vanishing criterion is proved in Subsection \ref{Subsec: vanishing criterion}.
\par
Section \ref{Sec: Orbit decomposition} is devoted to describing the orbit decomposition of the partial flag manifold under the action of Shalika subgroups. Furthermore, we calculate the conormal bundle and the modular characters which will occur in the proof of Theorem \ref{Main thmA}.\par
The complete proof of Theorem \ref{Main thmA} will be given in Section \ref{Sec: thmA}, where we also proved some homological finiteness results. Theorem \ref{Main thmB} is proved in Section \ref{Sec: thmB}. In Section \ref{Sec: linear model}, we related twisted Shalika periods to twisted linear periods using theta correspondence. 
In Section \ref{Sec: restrict to GL+}, we study the existence of the twisted Shalika periods on the representations of $\GL_{2n}^{+}(\BR)$.

\subsection{Conventions}\label{Subsec: conventions}
The notation $G_m\ (m \in \BZ_{\geq 1})$ stands for $\GL_m(\BK)$ through out the paper. When it is the situation that we should distinguish the case of $\BK = \BC$ and $\BK = \BR$, we shall use $\GL_m(\BC)$ and $\GL_m(\BR)$ instead.
\par
Lie groups are denoted by capital letters, and the modular character of Lie group G is denoted by $\delta_G$. Lie algebras of Lie groups are denoted by the corresponding Gothic letter, such as $\g = \Lie\ G$. Denote by $\gC := \g \otimes_{\BR} \BC$ the complexified Lie algebra of $G$. In general, for a $\BK$-vector space $V$, we use $V_{\BC} := V \otimes_{\BR} \BC$ to denote its complexification. 
\par
For a subgroup $H\subset G$ and a representation $\pi$ of $H$, take $g \in G$, we denote $H^g := g^{-1}Hg$, and $\pi^g$ a representation of $H^g$, defined by $\pi^g(g^{-1}hg) := \pi(h)$.
\par
For a Casselman-Wallach representation $\pi$, we write $\pi^{\vee}$ for the contragrediant Casselman-Wallach representation.
\par
Topological vector spaces are assumed to be Hausdorff unless otherwise specified.
\par
\subsection*{Acknowledgments}
The author would like to express his gratitude to his advisor, Prof. Binyong Sun, for providing invaluable guidance and offering continued support throughout the research. 
He is also grateful to Rui Chen, Hao Ying, and Zhihang Yu for fruitful discussions. He also thanks Wee Teck Gan, Dihua Jiang, Yubo Jin, Dongwen Liu, Hengfei Lu, and Ruichen Xu for their useful remarks.

\section{Preliminaries and notation}\label{Sec: preliminaries}
In this section, we recall some preliminary material and fix the notation to be used in the paper.

\subsection{Local Langlands correspondence}\label{Subsec: LLC}
In this subsection, we recall basic facts about the local Langlands correspondence for $\GL_{m}(\BK)\ (m \in \BZ_{\geq 1})$. For a general reference, see \cite{Knapp1994}, \cite{Moeglin1997}, \cite[Appendix]{Jacquet2009RankinSelberg}.
\par
We first review the Langlands classification of irreducible representations of $\GL_{m}(\BK)$. The characters of $\BC^{\times}$ can be uniquely written as 
$$\chi_{k,\lambda}^{\BC}(z) := (\frac{z}{|z|})^k |z|^{\lambda}, \quad k \in \BZ,\ \lambda \in \BC,$$ 
the characters of $\BR^{\times}$ can be uniquely written as 
$$\chi_{k,\lambda}^{\BR}(t) := \sgn^k(t) |t|^{\lambda}, \quad k \in \{0,1\},\ \lambda \in \BC.$$
For each character $\chi_{k,\lambda}^{\BK}$, we define its \textbf{exponent} by $$\exp(\chi_{k,\lambda}^{\BK}) := \Re\ \lambda.$$
We shall omit the superscript $\BK$ when the context is clear and use these notations for characters of $\BK^{\times}$ throughout the paper. 
\par
For each positive integer $k \in \BZ_{\geq 1}$, we write $D_k$ for the discrete series of $\GL_2(\BR)$ with central character $\sgn^{k+1}$, whose minimal $K$-type has highest weight $k+1$. Then for $\lambda \in \BC$, we denote $$D_{k,\lambda} := D_k \otimes |\det\nolimits_{\GL_2}|^{\lambda}.$$
The \textbf{exponent} of $D_{k,\lambda}$ is defined to be 
$$\exp(D_{k,\lambda}) := \Re\ \lambda.$$
Denote the set of relative discrete series of $\GL_2(\BR)$ by 
$$\Irr_{rd} (\GL_2(\BR)) := \{D_{k,\lambda}|\ k \in \BZ_{\geq 1}, \lambda \in \BC \}.$$
\par
Denote by $\Irr (\GL_m(\BK))$ the set of equivalence classes of irreducible Casselman-Wallach representations of $\GL_m(\BK)$. Define 
    $$
    \Irr_{rd}^{\BK} :=
    \left\{
    \begin{array}{ll} 
    \Irr (\GL_1(\BC)) & \text{if } \BK = \BC; \\
    \Irr (\GL_1(\BR)) \amalg \Irr_{rd} (\GL_2(\BR)) & \text{if } \BK = \BR.
    \end{array}
    \right.
    $$
\par
For $\pi_i \in \Irr_{rd}^{\BK}$, $i = 1,\dots, r$, the corresponding normalized parabolic induced representation 
$$\pi_1 \dot{\times} \cdots \dot{\times} \pi_r := \Ind^{G_m}_P \pi_1 \widehat{\otimes} \cdots \widehat{\otimes} \pi_r$$
is called an \textbf{generalized principal series representation}, where $P \subset G_m$ is a suitable parabolic subgroup. A generalized principal series representation is called a \textbf{induced representation of Langlands type} (or a \textbf{standard module}) if it satisfied 
$$\exp(\pi_1) \geq \dots \geq \exp(\pi_r).$$
The celebrated Langlands classification states that, for each irreducible representation $\pi$, there exists a unique standard module $\pi_1 \dot{\times} \cdots \dot{\times} \pi_r$ (up to isomorphism) such that $\pi$ is the unique irreducible quotient of $\pi_1 \dot{\times} \cdots \dot{\times} \pi_r$. If this is the case, we denote 
$$\pi =: \pi_1 \boxplus \cdots \boxplus \pi_r.$$ 
\par
Next, we recall the classification of irreducible representations of the Weil group $W_{\BK}$ of $\BK$. Recall that $W_{\BC} = \BC^{\times}$, and 
$W_{\BR} = \BC^{\times} \amalg \BC^{\times}j$ with $j^2 = -1$, $jzj^{-1} = \bar{z}$ for $z \in \BC^{\times}$. Since $\BC^{\times}$ is abelian, the irreducible representation of $W_{\BC}$ is just a character. As to $W_{\BR}$, consider the reciprocity map
$$
r:W_{\BR} \to \BR^{\times}
$$
defined by 
$$ j \mapsto -1, z \mapsto z\bar{z}. $$
It is surjective, and the kernel equals the derived subgroup of $W_{\BR}$. Thus we can view any one-dimensional representation of $W_{\BR}$ as a character of $\BR^{\times}$.
Any irreducible representation of $W_{\BR}$ is at most 2-dimensional. The 2-dimensional irreducible representation of $W_{\BR}$ can be uniquely written as 
$$\sigma_{k,\lambda} := \Ind_{\BC^{\times}}^{W_{\BR}} \chi^{\BC}_{k,2\lambda}, \quad k \in \BZ_{\geq 1},\ \lambda \in \BC.$$
\par
For $m \in \BZ_{\geq 1}$, denote by $\Phi(W_{\BK},m)$ the set of equivalence classes of $m$-dimensional semisimple complex representation of $W_{\BK}$. An element in $\Phi(W_{\BK},m)$ should be called an \textbf{L-parameter} for $\GL_{m}(\BK)$.
\par
Each irreducible representation of $\GL_m(\BK)$ is associated with an L-parameter.
Define a map $L$ from $\Irr (\GL_m(\BK))$ to $\Phi(W_{\BK},m)$ as following. For $\chi \in \Irr (\GL_1(\BK))$, define $L(\chi) \in \Phi(W_{\BK},1)$ via $\BC^{\times} = \BC^{\times}$ and the reciprocity map $r: W_{\BR} \to \BR^{\times}$. For $D_{k,\lambda} \in \Irr_{rd} (\GL_2(\BR))$, define $L(D_{k,\lambda}) := \sigma_{k,\lambda}$. Then using the Langlands classification, we define in general
$$L(\pi_1 \boxplus \cdots \boxplus \pi_r) := L(\pi_1) \oplus \cdots \oplus L(\pi_r).$$
The local Langlands correspondence says that $L$ is a bijection between $\Irr (\GL_m(\BK))$ and $\Phi(W_{\BK},m)$, with many other favorable properties.
\par
The following lemma describes the form for L-parameters of $\eta$-symplectic type. 
\begin{lemt}\label{Lem: symplectic parameter}
    An L-parameter $\phi: W_{\BK} \to \GL_{2n}(\mathbb{C)}$ is of $\eta$-symplectic type if and only if it has form 
    $$\phi = \sum_{1 \leq i \leq a} \phi_i + \sum_{a+1 \leq j \leq b} (\phi_j + \phi_j^{\vee} \cdot \eta),$$
    where $\phi_i\ (1 \leq i \leq a)$ are the two-dimensional representations of $\eta$-symplectic type, $\phi_j\ (a+1 \leq j \leq b)$ aren't of $\eta$-symplectic type.
\end{lemt}

\begin{proof}
    The direction of right to left is direct. As to the other direction, denote by $<,>$ the symplectic form on $\phi$. Let $\phi_1$ be an irreducible subrepresentation of $\phi$, which isn't of $\eta$-symplectic type. Since $\dim \phi_1 \leq 2$, we have $<,>|_{\phi_1 \times \phi_1} = 0$. Let $\phi_1^{\perp} := \{v \in \phi | <v,w> = 0, \forall w \in \phi_1\}$. Since $\phi$ is semisimple, there exist subrepresentation $\phi_2$ such that $\phi = \phi_1^{\perp} \oplus \phi_2$, also $\dim \phi_2 = \dim \phi_1$. Then $<,>|_{\phi_1 \times \phi_2}$ is non-degenerate, and $\phi_2 \cong \phi_1^{\vee} \otimes \eta$.
\end{proof}

We further recall the following criterion for generic representations, which will be used in the proof of Theorem \ref{Main thmB}.

\begin{lemt}[{\cite[Theorem 6.2]{Vogan1978gelfand}}, {\cite[Lemma 4.8]{fang2018godement}}]
    Every irreducible generic Casselman-Wallach representation of $\GL_n(\BK)$ is isomorphic to its standard module.
\end{lemt}

\subsection{Nuclear \fre  spaces}
In this subsection, we briefly recall some standard facts about nuclear \Fre spaces. See \cite[Appendix A]{Casselman2000Bruhat} for more details.
\par
Let $T: V \to W$ be a morphism of topological vector space, it is called \textbf{strict} if the induced map $V / \ker T \to \Ima T$ is a topological isomorphism.
\begin{dfnt}
    We call a sequence of (not necessarily Hausdorff) topological vector spaces
    \[
    \cdots \xrightarrow{} V_{i+1}  \xrightarrow{d_{i+1}} V_i  \xrightarrow{d_i} V_{i-1} \xrightarrow{} \cdots
    \]
    a \textbf{weakly exact sequence}, if the sequence is exact as vector space and all the morphisms are continuous. 
    Moreover, if all the morphisms are strict, we'll call the sequence a \textbf{strictly exact sequence}.
\end{dfnt}

\par
In this paper, we mainly consider nuclear \fre spaces, NF-spaces for short. If $W$ is a NF-space, $V \subset W$ is a closed subspace, then $V$ and $W/V$ are both NF-spaces. Furthermore, all surjective morphisms between \fre spaces are open.
\par
For two topological vector spaces $V$ and $W$, denote by $V \widehat{\otimes} W$ the completed projective tensor product of $V$ and $W$. If $V$ and $W$ are both NF-spaces, so is $V \widehat{\otimes} W$. The following lemma is useful in practice.

\begin{lemt}[{\cite[Lemma A.3]{Casselman2000Bruhat}}]\label{Lem: NF tensor exact}
    Let 
    \[
    0 \xrightarrow{} E_1  \xrightarrow{i} E_2  \xrightarrow{p} E_3 \xrightarrow{} 0
    \]
    be a strictly exact sequence of NF-spaces. Let $F$ be an NF-space. Then 
    \[
    0 \xrightarrow{} E_1 \widehat{\otimes} F \xrightarrow{i\widehat{\otimes} 1_F} E_2 \widehat{\otimes} F  \xrightarrow{p\widehat{\otimes} 1_F} E_3 \widehat{\otimes} F\xrightarrow{} 0
    \]
    is also a strictly exact sequence.
\end{lemt}

\section{Schwartz homology}\label{Sec: schwartz homology}
The theory of Schwartz homology is developed in \cite{CHEN2021108817}.
In this section, we review basic knowledge about Schwartz homology and present the propositions that will be used in our subsequent proofs. We refer readers to \cite{CHEN2021108817} for more details.

\subsection{Borel's lemma}

Let $M$ be a Nash manifold, and let $\CE$ be a tempered vector bundle over $M$. Denote $\Gamma^{\CS}(M,\CE)$ to be the Schwartz section of $\CE$ over $X$, which is a \fre space. Suppose that $U$ is an open Nash submanifold of $M$. Extension by zero yields a closed embedding $\Gamma^{\CS}(U,\CE) \hookrightarrow \Gamma^{\CS}(M,\CE)$. \par
Denote $Z:= M \backslash U$ and define  
$$\Gamma^{\CS}_Z(M,\CE) := \Gamma^{\CS}(M,\CE)/\Gamma^{\CS}(U,\CE).$$
Denote by $\CN_Z^{\ast}(M)$ the complexified conormal bundle of $Z$ in $M$. We have the following description of $\Gamma^{\CS}_Z(M,\CE)$.

\begin{prpt}[Borel's Lemma, {\cite[Proposition 8.2, 8.3]{CHEN2021108817}}, {\cite[Proposition 2.5]{Xue20Bessel}}]
There is a decreasing filtration on $\Gamma^{\CS}_Z(M,\CE)$, denoted by $\Gamma^{\CS}_Z(M,\CE)_k,\ k \in \BZ_{\geq0}$, such that
$$
\Gamma^{\CS}_Z(M,\CE) \cong  \varprojlim_k \, \Gamma_Z^{\CS}(M, \CE)/\Gamma_{Z}^{\CS}(M, \CE)_k
$$
as topological vector space, with the graded pieces
$$
 \Gamma_{Z}^{\CS}(M, \CE)_k/\Gamma_{Z}^{\CS}(M, \CE)_{k+1}\cong \Gamma^{\CS}(Z, \Sym^k(\CN^{\ast}_Z(M))\otimes \CE), \quad \forall k \in \BZ_{\geq0}.
$$
Moreover, let $G$ be an almost linear Nash group. If $M$ is a $G$-Nash manifold, $Z$ is stable under the action of $G$ and $\CE$ is a tempered $G$-bundle, then the filtration is stable under the action of $G$.
\end{prpt}

\subsection{Schwartz homology}
Let $G$ be an almost linear Nash group. Denote by $\CS\mathrm{mod}_G$ the category of smooth \fre representations of $G$ of moderate growth. \par
For $V \in \CS\mathrm{mod}_G$, consider the $G$-coinvariance $V_G := V/(\sum_{g \in G}(g-1)\cdot V)$, which is given the quotient topology and becomes a locally convex (not necessarily Hausdorff) topological vector space. The coinvariant functor $V \mapsto V_{G}$ is a right exact functor from $\CS\mathrm{mod}_G$ to the category of locally convex (not necessarily Hausdorff) topological vector spaces. In \cite{CHEN2021108817}, the Schwartz homology $\oH^\CS_i(G, -)$ is introduced to be the derived functor of $V \mapsto V_{G}$, which processes the following favorable properties.\par

\begin{prpt}[Shapiro's Lemma, {\cite[Theorem 7.5]{CHEN2021108817}}]\label{Prop: Shapiro Lem}
    Let $H$ be a Nash subgroup of an almost linear Nash group  $G$, and $V \in \CS\mathrm{mod}_H$.  Then there is an identification 
$$
\oH_{i}^\CS(G,  (\ind_H^G (V \otimes \delta_H))\otimes \delta_G^{-1})=\oH_{i}^\CS(H, V),\quad \forall\ i\in \mathbb Z, 
$$
of topological vector spaces, where $\ind^G_H$ denotes the Schwartz induction. 
When $G/H$ is compact, the Schwartz induction coincides with the unnormalized parabolic induction.
\end{prpt}

\begin{prpt}[{\cite[Theorem 7.7]{CHEN2021108817}}, {\cite[Lemma 3.1]{suzuki2023epsilon}}]\label{Prop: homology comparison}
For every representation $V$ in the category $\CS\mathrm{mod}_{G}$, there is an identification
\[
\oH^{\CS}_{i}(G, V) = \oH_{i}(\gC, K; V) = \oH_{i}(\gC, K; V^{K\text{-fin}}), \quad \forall i\in \BZ,
\]
of topological vector spaces. Here and henceforth, for $V \in \CS\mathrm{mod}_{G}$, $\oH_{*}(\gC, K; V)$ denotes the homology groups of the Koszul complex
\[
\cdots \rightarrow (\wedge^{l+1}(\gC/\k_{\BC}) \otimes V)_{K}
\rightarrow (\wedge^{l}(\gC/\k_{\BC}) \otimes V)_{K}
\rightarrow (\wedge^{l-1}(\gC/\k_{\BC}) \otimes V)_{K}
\rightarrow\cdots
\]
\end{prpt}

\begin{prpt}[{\cite[Corollary 7.8]{CHEN2021108817}}]
Every short exact sequence 
$$0\rightarrow V_{1}\rightarrow V_{2}\rightarrow V_{3}\rightarrow 0$$
in the category 
$\CS\mathrm{mod}_{G}$ yields a long exact sequence
\[
\cdots \rightarrow\oH^{\CS}_{i+1}(G, V_{3})\rightarrow\oH^{\CS}_{i}(G, V_{1})\rightarrow\oH^{\CS}_{i}(G, V_{2})\rightarrow \oH^{\CS}_{i}(G, V_{3})\rightarrow\cdots
\]
of (not necessarily Hausdorff)  locally convex topological vector spaces.
\end{prpt}

In order to use Borel's lemma to study the Schwartz homology, we also need the following proposition.
\begin{prpt}[{\cite[Lemma 8.4]{CHEN2021108817}}]\label{Prop: homology borel lem}
Suppose that  $Z$ is a $G$-stable closed Nash submanifold of a $G$-Nash manifold $M$. Fix a character $\chi$ of $G$.
Let $i\in \BZ$ and assume that $\oH_{i+1}^\CS(G;(\Gamma_{Z}^{\varsigma}(M, \mathsf E)/\Gamma_{Z}^{\varsigma}(M, \mathsf E)_k)\otimes\chi)$ is finite-dimensional for all $k\geq 0$. Then the canonical map
\[
\oH_{i}^\CS(G;\Gamma_{Z}^{\varsigma}(M, \mathsf E)\otimes\chi)\rightarrow\varprojlim_{k}\oH_{i}^\CS(G;(\Gamma_{Z}^{\varsigma}(M, \mathsf E)/\Gamma_{Z}^{\varsigma}(M, \mathsf E)_k)\otimes\chi)
\]
is a linear isomorphism.
\end{prpt}

It is important in practice to determine the Hausdorffness of the Schwartz homology. We have the following two propositions which are useful.

\begin{prpt}[{\cite[Theorem 5.9, Proposition 5.7]{CHEN2021108817}}]\label{Prop: projective hausdorff}
Let $G$ be an almost linear Nash group, and let $V$ be a relatively projective representation in $\CS\mathrm{mod}_G$.  Then the coinvariant space $V_G$ is a \Fre space.
Moreover, when $G$ is compact, every representation in $\CS\mathrm{mod}_G$ is relatively projective. 
\end{prpt}

\begin{prpt}[{\cite[Lemma 3.4]{borel2000continuous}}, {\cite[Proposition 1.9]{CHEN2021108817}}]\label{Prop: finite hausdorff}
Let $G$ be an almost linear Nash group, and let $V \in \CS\mathrm{mod}_G$.  If $\oH_{i}^\CS(G; V)$ is finite-dimensional ($i\in \BZ$), then it is Hausdorff.  
\end{prpt}

Since $\CS\mathrm{mod}_G$ is not an Abelian category, the usual homological tools cannot be applied directly.
We prove the following two useful formulas for the Schwartz homology, with their proofs deferred to Appendix \ref{Sec: Appendix A}.

\begin{prpt}[\kun formula, Theorem \ref{Thm: kunneth Schwartz homology}]
    Let $G_1$ and $G_2$ be two almost linear Nash groups.
    Assume $V_i \in \CS\mathrm{mod}_{G_i}$, $i = 1,2$, are both NF-spaces. If $\forall\ j \in \BZ$, $\oH^{\CS}_{j}(G_i, V_i)$ are NF-spaces, then 
    $$
    \oH^{\CS}_m(G_1 \times G_2, V_1 \widehat{\otimes} V_2) \cong \bigoplus_{p+q=m} \oH^{\CS}_p(G_1, V_1) \widehat{\otimes} \oH^{\CS}_q(G_2, V_2), \quad \forall m \in \BZ,
    $$
    as topological vector spaces. In particular, $\oH^{\CS}_m(G_1 \times G_2, V_1 \widehat{\otimes} V_2)$ are NF-spaces. 
\end{prpt}

\begin{rremark}
    This \kun formula cannot be derived from that of the relative Lie algebra homology, 
    as we do not assume the representations to be admissible.
\end{rremark}

\begin{prpt}[Hochschild-Serre spectral sequence for nilpotent normal subgroups, Corallary \ref{Cor: spec seq nilp}]
    Let $H = L \ltimes N$ be an almost linear Nash group with $N$ nilpotent normal subgroup. Consider $V \in \CS\mathrm{mod}_H$, if $\forall\ j \in \BZ$, $\oH^{\CS}_{j}(N, V) \in \CS\mathrm{mod}_L$, then there exists convergent first quadrant spectral sequences:
    $$
    E^2_{p,q} = \oH^{\CS}_{p}(L, \oH^{\CS}_{q}(N, V)) \implies \oH^{\CS}_{p+q}(H ,V).
    $$
\end{prpt}

\subsection{Some vanishing criterion}\label{Subsec: vanishing criterion}
In this section, we state two vanishing criteria for the Schwartz homology groups, which will be frequently used in the subsequent sections.
\begin{lemt}\label{Lem: nilp vanish lem}
    Let $\n$ be a nilpotent complex Lie algebra, and let $\psi$ be a non-trivial character on $\n$. Denote by $V$ a (probably infinite-dimensional) trivial representation of $\n$. Then $\oH_{i}(\n, V\otimes \psi) = 0,\ \forall i \in \BZ$.
\end{lemt}

\begin{proof}
    When $\dim \n = 1$, then the homology group can be computed by
    \begin{equation*}
    \begin{array}{rlccl}
        0 \to & \n \otimes V \otimes \psi& \to &V \otimes \psi& \to 0\\
        &X\otimes v& \mapsto& \psi(X)v&
    \end{array}.
    \end{equation*}
    Since $\psi$ is nontrivial, the middle map is a bijective map. Thus 
    $$\oH_{i}(\n, V\otimes \psi) = 0,\ \forall i \in \BZ.$$
    \par
    Now we assume $\n$ is abelian, then there exist ideal $\n_0 \subset \n$, with $\dim \n_0 = 1$, s.t. $\psi|_{\n_0} \neq 0$. Thus we have $\oH_{i}(\n_0, V\otimes \psi) = 0,\ \forall i \in \BZ$. Following from spectral sequence, we have $\oH_{i}(\n, V\otimes \psi) = 0,\ \forall i \in \BZ$.\par
    For a nilpotent lie algebra $\n$, let $\n^{(1)} := [\n,\n]$, then $\n/\n^{(1)}$ is abelian and $\psi|_{\n^{(1)}} = 0$, thus $\psi$ can be seen as a nontrivial character on $\n/\n^{(1)}$. Note that
    $$
    \oH_{i}(\n/\n^{(1)}, \oH_{j}(\n^{(1)}, V\otimes \psi)) \cong  \oH_{i}(\n/\n^{(1)}, \oH_{j}(\n^{(1)}, \BC)\otimes V\otimes \psi)).
    $$
    Since $\n/\n^{(1)}$ acts on $\n^{(1)}$ algebraically, $\oH_{j}(\n^{(1)}, \BC)$ admits a finite filtration with each grading piece trivial action. Then using the abelian case, we have 
    $$\oH_{i}(\n/\n^{(1)}, \oH_{j}(\n^{(1)}, V\otimes \psi))=0,$$
    thus $\oH_{i}(\n, V\otimes \psi) = 0,\ \forall i \in \BZ$.
\end{proof}

\begin{lemt}\label{Lem: center vanish lem}
    Let $V \in \CS\mathrm{mod}_G$. Denote the center of $G$ by $C(G)$.
    Assume $\exists\ a \in C(G)$, which acts on $V$ by a scalar $c \neq 1$. Then $\oH^{\CS}_{i}(G, V) = 0,\ \forall\ i \in \BZ.$
\end{lemt}

\begin{proof}
    Since $a \in C(G)$, then $\phi := \pi(a) \in \Hom_G(V, V)$. Also $\phi(v) = c \cdot v$, thus $\phi : \oH^{\CS}_{i}(G, V) \to \oH^{\CS}_{i}(G, V)$ maps by the scalar $c$. On the other hand, consider $P_. \twoheadrightarrow V$ a relative strong projective resolution of $V$. Then $\pi_i(a): P_i \to P_i$ gives a lifting of $\phi$ to a morphism of chain complexes. Note that $\pi_i(a) : (P_i)_G \to (P_i)_G$ descent to identity maps. Thus $\phi$ induce identity map on $\oH^{\CS}_{i}(G, V)$. Since $c \neq 1$, we have $\oH^{\CS}_{i}(G, V) = 0,\ \forall\ i \in \BZ$.
\end{proof}

\section{Orbit decomposition}\label{Sec: Orbit decomposition}
In this section, we describe the orbit decomposition of the partial flag variety with respect to the Shalika subgroup. Furthermore, we establish certain numerical lemmas, preparing for the calculation in the following sections.
\par
We first give some general notations on root systems.
For $m \in \BZ_{\geq 1}$, let $\t_{m}$  be the Lie subalgebra of diagonal matrix in $\gl_{m}$. Define $e_i$ to be an element in $\t_{m}^*$ by
\[
e_i\left( 
\begin{bmatrix}
\begin{smallmatrix}
x_{1} & & \\ 
 & \ddots & \\
&& x_{m}
\end{smallmatrix}
\end{bmatrix} \right) 
= x_i
\]
and denote $\alpha_{i.j} := e_i - e_j$ for $1 \leq i \neq j \leq m$. Then the set of roots of $\gl_{m}$ with respect to $\t_m$ is
$$\Phi_{m} := \{\alpha_{i,j},\ 1 \leq i \neq j \leq m\},$$
with positive roots
$$\Phi_{m}^{+} := \{\alpha_{i,j},\ 1 \leq i < j \leq m\}.$$
We simply denote $\alpha_i := \alpha_{i,i+1}$ for $i \in \{1,\dots, m-1\}$, then $\Delta_m := \{\alpha_1, \alpha_2,...,\alpha_{m-1}\}$ is the set of simple roots. For each standard parabolic subgroup $Q \in G_{m}$, denote by $\Delta_Q$ the subset of $\Delta_m$ corresponding to $Q$. 
\par
Denote $W_{m}$ to be the Weyl group of $\gl_{m}$ with respect to $\t_m$, which is just the symmetric group $\fS_{m}$. We write the permutations in $\fS_m$ as 
$$
\omega = (i_1,\dots , i_m),
$$
where $i_k = \omega(k), k = 1,\dots,m$, and identify $\fS_m$ with the group of $m\times m$ permutation matrices such that
$$
\omega E_{j,k} \omega^{-1} = E_{i_j,i_k}, \quad j,k = 1,2,\dots,n,
$$
Where $E_{j,k}$ denotes the standard $m\times m$ elementary matrices. From now on, the size of $E_{j,k}$ will be $2n$ whenever we use this notation except in Section \ref{Sec: linear model}. For instance the following Weyl element $\omega_{\iota}$ in $\fS_{2n}$ corresponding to the matrix
$$
\omega_{\iota} = (n+1,n+2,\dots,2n,1,2,\dots,n) \leftrightarrow 
\begin{bmatrix}
0 & 1_n \\ 1_n & 0 
\end{bmatrix}.
$$
\par
For two standard parabolic subgroup $P_1, P_2$ of $G_m$, we define the following subset of $W_m$,
$$
\prescript{P_1}{}{W_{m}}^{P_2} := \left\{ \omega \in W_m \mid \omega(\Delta_{P_2}) \subset \Phi_{m}^+,\ \omega^{-1}(\Delta_{P_1}) \subset \Phi_{m}^+
    \right\},
$$
 which will occur in the relative Bruhat decomposition.
\par

Let $Q$ be the standard parabolic subgroup in $G_{2n}$ corresponding to the partition $(n,n)$, with Levi decomposition $Q =M\ltimes U$. Denote by $S$ the Shalika subgroup of $G_{2n}$. Note that $S$ is a subgroup of $Q$, with Levi decomposition $S = G^{\Delta}_n \ltimes U$. Here, $G^{\Delta}_n$ is the diagonal embedding of $G_n$ into $M = G_n \times G_n$.
\par

\begin{prpt}\label{Prop: orbit decomposition}
    Let $P$ be a standard parabolic subgroup in $G_{2n}$, and let $Q$ be as defined above. Then we have the following bijection
    $$P\backslash G_{2n} / S \xleftrightarrow{1:1} \Omega := \{\omega \in W_{2n} |\ \omega = \gamma \cdot \sigma, \gamma \in  \prescript{P}{}{W_{2n}}^Q, \sigma \in  \prescript{P_{1,\gamma}}{}{W_n}^{P_{2,\gamma}}\},$$
    where for $\gamma \in W_{2n}$, $P^{\gamma} \cap (G_n \times G_n) = P_{1,\gamma} \times P_{2,\gamma}$. Here, we viewed $W_n$ as a subgroup of $W_{2n}$ by interchanging the first $n$ variable. \par
\end{prpt}

\begin{proof}
    Following the relative Bruhat decomposition, we have
    $$
    P\backslash G_{2n} / Q \xleftrightarrow{1:1} \prescript{P}{}{W_{n}}^Q.
    $$
    Since $\gamma \in W_{2n}$, $P^{\gamma} \cap (G_n \times G_n)$ is a parabolic subgroup in $G_n \times G_n$, denote it by $P^{\gamma} \cap (G_n \times G_n) = P_{1,\gamma} \times P_{2,\gamma}$.
    Then we have the following bijections
    $$
    \begin{aligned}
    P^{\gamma} \cap Q \backslash Q / S &\xleftrightarrow{1:1}
    (P^{\gamma} \cap (G_{n} \times G_{n}))\backslash G_{n} \times G_{n} / G_n^{\Delta}\\
    &\xleftrightarrow{1:1}
    P_{1,\gamma}\backslash G_{n} / P_{2,\gamma}\\
    &\xleftrightarrow{1:1} \prescript{P_{1,\gamma}}{}{W_n}^{P_{2,\gamma}}
    \end{aligned}.
    $$
    The first bijection follows from $P^{\gamma} \cap Q = (P^{\gamma} \cap M) \ltimes (P^{\gamma} \cap U)$, the third bijection follows from the relative Bruhat decomposition and the second bijection is induced from
    $$
    \begin{array}{rcl}
    P_{1,\gamma}\backslash G_{n} / P_{2,\gamma} &\to& (P^{\gamma} \cap (G_{n} \times G_{n}))\backslash G_{n} \times G_{n} / G_n^{\Delta}\\
    g &\mapsto&
    \begin{bmatrix} g & 0 \\ 0 & 1_n \end{bmatrix}
    \end{array}.
    $$
    \par
    The orbit decomposition follows from these two identifications.
\end{proof}

Since then, for each $\omega \in \Omega$, we denote $\CO_{\omega}$ the corresponding $S$-orbit in $P\backslash G_{2n}$.
\par
In the rest of this section, we specialize to the case when $P \subset G_{2n}$ is a standard cuspidal parabolic subgroup corresponding to the partition $(n_1, \dots, n_r)$ of $2n$. For $\BK = \BC$, we have $n_i = 1,\ \forall\ i \in \{1,\dots,r=2n\}$, thus $P$ is just the standard Borel subgroup in $\GL_{2n}(\BC)$. For $\BK = \BR$, we have $n_i \in \{1,2\} ,\ \forall\ i \in \{1,\dots,r\}$. Denote the Levi decomposition of $P$ by $P = L \ltimes N$. 

We shall present some specific calculations which will be used in the subsequent sections. Define the following subset of $\Phi_{2n}$: 
$$\Phi_n^{(1)} := \{\alpha_{i,j},\ 1 \leq i \neq j \leq n\},\ \Phi_n^{(2)} := \{\alpha_{n+i,n+j},\ 1 \leq i \neq j \leq n\},$$
and $\Phi_n^{(k),+} := \Phi_n^{(k)} \cap \Phi_{2n}^{+},\ k = 1,2$.

\begin{lemt}
    For $\gamma \in \prescript{P}{}{W_{2n}}^Q$, $P_{1,\gamma}$ $(\text{resp. } P_{2,\gamma})$ as in Proposition \ref{Prop: orbit decomposition} is a standard cuspidal paraboic subgroup in $G_n$, with corresponding simple roots $\gamma^{-1}(\Delta_P) \cap \Phi_n^{(1)} = \gamma^{-1}(\Delta_P) \cap \{\alpha_1,\dots,\alpha_{n-1}\}$ $(\text{resp. } \gamma^{-1}(\Delta_P) \cap \Phi_n^{(2)} = \gamma^{-1}(\Delta_P) \cap \{\alpha_{n+1},\dots,\alpha_{2n-1}\})$.  
\end{lemt}

\begin{proof}
    We prove for $P_{1,\gamma}$ in the following, the case of $P_{2,\gamma}$ is similar. Let $\Gamma := \Phi_{2n}^{+} \amalg -\Delta_P$. Then the roots of $\p_{1,\gamma}$ is $\Gamma' := \gamma^{-1}(\Gamma) \cap \Phi^{(1)}_n$. Note that $\gamma^{-1}(\Phi_{2n}^{+}) \cap \Phi^{(1)}_n = \Phi^{(1),+}_n$, thus $P_{1,\gamma}$ is a standard parabolic subgroup.\par
    Assume $\alpha_{k,l} \in \gamma^{-1}(\Delta_P) \cap \Phi_n^{(1)}$, thus $k < l$ and $\gamma(l) = \gamma(k) + 1$. Since $\gamma(1) < \cdots < \gamma(n)$, we have $l = k+1$. Thus $\gamma^{-1}(\Delta_P) \cap \Phi_n^{(1)} = \gamma^{-1}(\Delta_P) \cap \{\alpha_1,\dots,\alpha_{n-1}\}$. 
    Assume $\alpha_a, \alpha_{a+1} \in \gamma^{-1}(\Delta_P) \cap \{\alpha_1,\dots,\alpha_{n-1}\}$ for some $1\leq a \leq n-2$, then $\alpha_{\gamma(a),\gamma(a+1)}, \alpha_{\gamma(a+1),\gamma(a+2)} \in \Delta_P$, contradict with $P$ is a cuspidal parabolic. Thus $P_{1,\gamma}$ is a standard cuspidal parabolic subgroup.    
\end{proof}

We classify these orbits into different types and treat them separately.

\begin{dfnt}
    An orbit $\CO_{\omega}$ is called \textbf{$\psi$-vanishing}, if $\xi_{\eta,\psi} |_{N^{\omega} \cap U} \neq Id $. It is called \textbf{$\psi$-unvanishing} otherwise. We shall say $\omega$ is $\psi$-vanishing (resp. $\psi$-unvanishing) for short.
\end{dfnt}

The following two lemmas give numerical conditions of the $\psi$-unvanishing orbits.

\begin{lemt}
    $\omega$ is $\psi$-unvanishing if and only if $\forall\ k \in \{1,\dots,n\}$, we have $\omega(k) > \omega(n+k)$ or $\alpha_{\omega(k),\omega(n+k)} \in \Delta_P$.
\end{lemt}

\begin{proof}
    Note that $\xi_{\eta,\psi} |_{N^{\omega} \cap U} = Id $ is equivalent to $\forall\ 1 \leq i \leq n, \alpha_{i,n+i} \notin \n^{\omega}$. Then the lemma follows from the structure of the roots of $\n$.
\end{proof}

\begin{lemt}
    Assume $\omega$ is $\psi$-unvanishing. For $\beta \in \Delta_p$, if $\omega^{-1}(\beta) \in \{\alpha_{i,n+j}, 1 \leq i,j \leq n\}$, then $\omega^{-1}(\beta) \in \{\alpha_{i,n+j}, 1 \leq i \leq j \leq n\}$.
\end{lemt}

\begin{proof}
    Take $\beta = \alpha_a$ such that $\omega^{-1}(\alpha_a) \in \{\alpha_{i,n+j}, 1 \leq i,j \leq n\}$, thus $1 \leq \omega^{-1}(a) \leq n$, $\omega^{-1}(a+1) \geq n+1$. We prove the lemma by contradiction.
    
    If $\omega^{-1}(a) + n > \omega^{-1}(a+1)$, then $\omega(\omega^{-1}(a)+n) = \gamma(\omega^{-1}(a)+n) > \gamma(\omega^{-1}(a+1)) = \omega(\omega^{-1}(a+1)) =a+1$. On the other hand, since $\omega$ is $\psi$-unvanishing, then we have $a = \omega(\omega^{-1}(a)) > \omega(\omega^{-1}(a)+n)$ or $\omega(\omega^{-1}(a)+n) = a+1$. In the first case, we obtain $a > \omega(\omega^{-1}(a)+n) > a+1$, contradicts. In the second case, we obtain $\omega^{-1}(a) + n = \omega^{-1}(a+1)$, contradicts.
\end{proof}

For $\omega \in \Omega$ as in Proposition \ref{Prop: orbit decomposition}, we introduce the following notation, which will be helpful for the discussion below. \par
\begin{itemize}
    \item $\Psi_{P,\omega} := \omega^{-1}(\Delta_P)$,
    \item $\Psi^{ma}_{P,\omega} := \Psi_{P,\omega} \cap \omega_{\iota}(\Psi_{P,\omega})$, where $\omega_{\iota} = (n+1,n+2,\dots,2n,1,2,\dots,n)$,
    \item $\Psi^{wh}_{P,\omega} := \Psi_{P,\omega} \cap \{\alpha_{i,n+i},\ i\in\{1,\dots,n\}\}$, 
    \item $\Psi^{um}_{P,\omega} := \Psi_{P,\omega} \backslash (\Psi^{ma}_{P,\omega} \amalg \Psi^{wh}_{P,\omega})$,
    \item $\Delta^{*}_{P,\omega} := \omega(\Psi^{*}_{P,\omega})$, for $* \in \{ma,\ wh,\ um\}$
\end{itemize}

\begin{dfnt}\label{Def: matching orbit}
    A $\psi$-unvanishing orbit $\CO_{\omega}$ is called a \textbf{matching orbit}, if 
    $$\Psi^{um}_{P,\omega} = \varnothing.$$
    Otherwise, it is called an \textbf{unmatching orbit}. We shall say $\omega$ is matching (resp. unmatching) for short.
\end{dfnt}

We introduce the following notation which will be used in the rest of the paper. Define $a_1 := 1$. For $i \in \{2,3,\dots,r\}$, let $a_i := 1 + \sum_{j=1}^{i-1} n_j$. If $n_i = 2$, let $b_i := a_i +1$. Then we have $\{a_i\} \amalg \{b_i\} = \{1,\cdots,2n\}$.\par
We may briefly illustrate why the matching orbits are significant. For each matching orbit $\CO_{\omega}$, we can attach an element $s_{\omega}$ in $\fS_r$ as follows. If $\omega(k) = a_i$, $\omega(n+k) = a_j$, then we set $s_{\omega}(i) = j$, $s_{\omega}(j) = i$. If $\omega(k) = a_i$, $\omega(n+k) = b_i$, then we set $s_{\omega}(i) = i$ (in this case $n_i = 2$). Such $s_{\omega}$ is well-defined under the condition that $\omega$ is matching. It in fact gives a partition of $\{1,\dots,r\}$ into pairs and singleton, which corresponds to the symplectic condition of the L-parameters.
\par

In order to use the argument of spectral sequence, we need the following group decomposition of $P^{\omega} \cap S$, which is slightly different from the Levi decomposition of $P^{\omega} \cap S$.
\par
Since $\omega$ is a Weyl element, we have $P^{\omega} \cap Q = P^{\omega} \cap M \cdot P^{\omega} \cap U $. Thus $P^{\omega} \cap S = P^{\omega} \cap G_n^{\Delta} \cdot P^{\omega} \cap U $. Denote $P^{\omega} \cap M = P_{1,\omega} \times P_{2,\omega}$, then $P^{\omega} \cap G_n^{\Delta} = (P_{1,\omega} \cap P_{2,\omega})^{\Delta}$. Denote the Levi decomposition of $P_{1,\omega} \cap P_{2,\omega} := A_{\omega} \ltimes B_{\omega}$.
\par
Also note that $P^{\omega} \cap U =  L^{\omega} \cap U \ltimes  N^{\omega} \cap U$. Denote by $U^{\dag}$ the subgroup of $U$ generated by $\{1+E_{i,n+i}|\ i=1,\dots,n\}$, split $P^{\omega} \cap U$ into two parts $P^{\omega} \cap U =  L^{\omega} \cap U^{\dag} \ltimes  C_{\omega}$. Then we define $R_{\omega} := A_{\omega}^{\Delta} \cdot (L^{\omega} \cap U^{\dag})$, $V_{\omega} := B_{\omega}^{\Delta} \cdot C_{\omega}$.

\begin{lemt}\label{Lem: gp decomp}
For a $\psi$-unvanishing $S$-orbit $\CO_{\omega}$, we have
$$
P^{\omega} \cap S = R_{\omega} \ltimes V_{\omega}.
$$
\end{lemt}

\begin{proof}
    It suffices to prove that $L^{\omega} \cap U^{\dag}$ stable $V_{\omega}$. Note that $B_{\omega}$ consists of strictly block upper triangular matrices, and $L^{\omega} \cap U^{\dag}$ has the same block type with $B_{\omega}$. Then the lemma is a direct calculation.
\end{proof}

The following lemma shows that if $\alpha$ is a simple root of $\p$, then either $\omega^{-1}(\alpha)$ or $\omega^{-1}(\alpha) + \omega_{\iota}(\omega^{-1}(\alpha))$
is contained in roots of $\p^{\omega} \cap \s$.

\begin{lemt}\label{Lem: simple roots alive}
If $\alpha_{n+k} \in \omega^{-1}(\Delta_{P})$ for some $1 \leq k \leq n-1$, then $\omega(\alpha_{k}) \in \Phi_{2n}^{+}$.
\end{lemt}

\begin{proof}
    Denote $\omega = \gamma \cdot \sigma \in \Omega$ and
    $\beta := \omega(\alpha_{n+k}) \in \Delta_{P}$, then $\gamma^{-1}(\beta) = \alpha_{n+k}$. Note that $\sigma$ satisfies $\sigma(\omega_{\iota}(\gamma^{-1}(\Delta_P) \cap \{\alpha_{n+1}, \dots, \alpha_{2n-1}\})) \in \Phi_n^{(1),+}$, thus $\sigma(\alpha_k) \in \Phi_n^{(1),+}$. Following the property of $\gamma \in \prescript{P}{}{W_{2n}}^Q$, we obtain $\omega(\alpha_k) \in \Phi_{2n}^{+}$. 
\end{proof}

The idea is rather simple, although the formulation is abstract. Let us present an example in $\GL_6(\BR)$ to illustrate the idea. Readers are encouraged to work through this example to gain a better understanding.
Note that for the case of $\GL_{2n}(\BC)$, we have $\Delta_P = \varnothing$. Thus matching orbits are just $\psi$-unvanishing $S$-orbits, and all the formulation above is almost trivial.

\begin{example}
    Let $P = LN$ be the standard parabolic subgroup of $\GL_6(\BR)$ with Levi subgroup $L = \GL_2(\BR) \times \GL_2(\BR) \times \GL_2(\BR)$. There are only 5 $\psi$-unvanishing $S$-orbits, which are represented by 
    $\{\sigma_1 = (1,3,5,2,4,6), \sigma_2 =(1,5,6,2,3,4), \sigma_3 =(3,4,5,1,2,6), \sigma_4 =(5,6,3,1,2,4), \sigma_5 =(3,5,6,1,2,4) \}$. One can check that, in this case, the only unmatching orbit is $\CO_{\sigma_5}$, while the other four orbits are all matching orbits. Now we use $\sigma_3$ and $\sigma_5$ to show what happens on group decomposition. For the case of $\CO_{\sigma_3}$, we have
    \[
    P^{\sigma_3} = 
    \begin{bmatrix} 
    \begin{smallmatrix}
    *&*&-&0&0&- \\
    *&*&-&0&0&- \\
    0&0&*&0&0&* \\
    -&-&-&*&*&- \\
    -&-&-&*&*&- \\
    0&0&*&0&0&*
    \end{smallmatrix}
    \end{bmatrix},\ \ 
    P^{\sigma_3} \cap S = 
    \begin{bmatrix}
    \begin{smallmatrix}
    a_1&a_2&n_1&0&0&- \\
    a_3&a_4&n_2&0&0&- \\
    0&0&b_1&0&0&* \\
    0&0&0&a_1&a_2&n_1 \\
    0&0&0&a_3&a_4&n_2 \\
    0&0&0&0&0&b_1
    \end{smallmatrix}
    \end{bmatrix}.
    \]
    Here and follows, both $*$ and $-$ denote arbitrary elements with $*$ contains in $L^{\sigma_i}$ and $-$ contains in $N^{\sigma_i}$. The letters $a_1,\dots,b_1,\dots,n_1,\dots$ denote arbitrary elements, but the two appearances of the same letter denote the same numbers.\par
    In this case the group decomposition in Lemma \ref{Lem: gp decomp} is just
    \[
    R_{\sigma_3} = 
    \begin{bmatrix} 
    \begin{smallmatrix}
    a_1&a_2&0&0&0&0 \\
    a_3&a_4&0&0&0&0 \\
    0&0&b_1&0&0&* \\
    0&0&0&a_1&a_2&0 \\
    0&0&0&a_3&a_4&0 \\
    0&0&0&0&0&b_1
    \end{smallmatrix}
    \end{bmatrix},\ \ 
    V_{\sigma_3}  = 
    \begin{bmatrix} 
    \begin{smallmatrix}
    1&0&n_1&0&0&- \\
    0&1&n_2&0&0&- \\
    0&0&1&0&0&0 \\
    0&0&0&1&0&n_1 \\
    0&0&0&0&1&n_2 \\
    0&0&0&0&0&1
    \end{smallmatrix}
    \end{bmatrix}.
    \]
    For the unmatching orbit $\CO_{\sigma_5}$, we have
    \[
    P^{\sigma_5} = 
    \begin{bmatrix} 
    \begin{smallmatrix}
    *&-&-&0&0&* \\
    0&*&*&0&0&0 \\
    0&*&*&0&0&0 \\
    -&-&-&*&*&- \\
    -&-&-&*&*&- \\
    *&-&-&0&0&*
    \end{smallmatrix}
    \end{bmatrix},\ \ 
    P^{\sigma_5} \cap S = 
    \begin{bmatrix} 
    \begin{smallmatrix}
    a_1&n_1&n_2&0&0&* \\
    0&a_2&n_3&0&0&0 \\
    0&0&a_3&0&0&0 \\
    0&0&0&a_1&n_1&n_2 \\
    0&0&0&0&a_2&n_3 \\
    0&0&0&0&0&a_3
    \end{smallmatrix}
    \end{bmatrix}.
    \]
    In this case the group decomposition in Lemma \ref{Lem: gp decomp} is just
    \[
    R_{\sigma_5} = 
    \begin{bmatrix} 
    \begin{smallmatrix}
    a_1&0&0&0&0&0 \\
    0&a_2&0&0&0&0 \\
    0&0&a_3&0&0&0 \\
    0&0&0&a_1&0&0 \\
    0&0&0&0&a_2&0 \\
    0&0&0&0&0&a_3
    \end{smallmatrix}
    \end{bmatrix},\ \ 
    V_{\sigma_5}  = 
    \begin{bmatrix} 
    \begin{smallmatrix}
    1&n_1&n_2&0&0&* \\
    0&1&n_3&0&0&0 \\
    0&0&1&0&0&0 \\
    0&0&0&1&n_1&n_2 \\
    0&0&0&0&1&n_3 \\
    0&0&0&0&0&1
    \end{smallmatrix}
    \end{bmatrix}.
    \]
\end{example}

In order to use Borel's lemma, we need the following two numerical lemmas about the representative elements of the conormal bundle.

\begin{lemt}
    Let $\CO_{\omega}$ be a $\psi$-unvanishing orbit.
    Denote 
    \begin{multline*}
    \CN_{\omega} := \Span_{\BK}(\{E_{n+j,i} | 1\leq i < j \leq n,\ \omega(i) < \omega(n+j),\ \alpha_{i, n+j} \notin \omega^{-1}(\Delta_{P})\} \\ 
    \amalg \{E_{j,i} | 1 \leq i < j \leq n,\ \omega(i) < \omega(j),\ \alpha_{i,j} \notin \omega^{-1}(\Delta_{P}),\ \alpha_{n+i,n+j} \notin \omega^{-1}(\Delta_{P})\}).
    \end{multline*}
    Then we have 
    $$
    \gl_{2n} = \CN_{\omega} \oplus (\p^{\omega} + \s).
    $$
\end{lemt}
\begin{proof}
    Denote by $\bar{\u}$ (resp. $\bar{\n}$) the transpose of $\u$ (resp. $\n$). For $i,j \in \{1,\dots,n\}$, if $E_{n+j,i} \in \p^{\omega} + \s$, then $E_{n+j,i} \in \p^{\omega} \cap \bar{\u}$. Note that $\bar{\u} = \p^{\omega} \cap \bar{\u} \oplus \bar{\n}^{\omega} \cap \bar{\u}$, thus $E_{n+j,i} \notin \p^{\omega} + \s \Leftrightarrow E_{n+j,i} \in \bar{\n}^{\omega} \cap \bar{\u} \Leftrightarrow E_{i,n+j} \in \n^{\omega} \cap \u$. Since $\omega$ is $\psi$-unvanishing, we can directly calculate that 
    $\n^{\omega} \cap \u = \Span_{\BK} \{E_{i,n+j} | 1 \leq i < j \leq n, \omega(i) < \omega(n+j), \omega(\alpha_{i,n+j}) \notin \Delta_P \} $.
    \par
    Since $E_{j,i} \notin p^{\omega} + \s \Leftrightarrow E_{j,i} \notin p^{\omega} \cap \gl_n^{(1)} \text{ and } E_{n+j,n+i} \notin p^{\omega} \cap \gl_n^{(2)}$, one can directly calculate that
    $E_{j,i} \notin p^{\omega} + \s \Leftrightarrow i<j, \omega(i) < \omega(j), \alpha_{i,j} \notin \omega^{-1}(\Delta_{P}), \alpha_{n+i,n+j} \notin \omega^{-1}(\Delta_{P})$.
    \par
    Then the lemma follows from the above two calculations.
\end{proof}

\begin{lemt}\label{Lem: Der associated positive exp}
    Let $\CO_{\omega}$ be a $\psi$-unvanishing orbit. 
    \begin{enumerate}
        \item If $E_{n+j,i} \in \CN_{\omega}$, then we have $\omega(i) < \omega(j)$, $\omega(n+i) < \omega(n+j)$.
        \item If $E_{j,i} \in \CN_{\omega}$, then we have $\omega(i) < \omega(j)$, $\omega(n+i) < \omega(n+j)$.
    \end{enumerate}
\end{lemt}

\begin{proof}
    The second assertion follows from the $\psi$-unvanishing condition. As to the first one, we have 
    $\omega(i) > \omega(n+i)$ or $\omega(i) + 1 = \omega(n+i)$. Since $\omega(i) < \omega(n+j)$ and $i<j$, we obtain $\omega(n+i) < \omega(n+j)$ in both case. Similar argument gives that $\omega(i) < \omega(j)$.
\end{proof}

In the last of this section, we study the modular characters that will occur in Shapiro's lemma in the calculation of the next section. Denote
\begin{itemize}
    \item $\Lambda^{ma}_{P,\omega} := \{1\leq i \leq n |\ \alpha_i \in \Psi^{ma}_{P,\omega} \}$, 
    \item $\Lambda^{um}_{P,\omega} := \{(i,j)|\ 1\leq i<j \leq n ,\ \alpha_{i,j}\ \text{or}\ \alpha_{i,n+j}\ \text{or}\ \alpha_{n+i,n+j}   \in \Psi^{um}_{P,\omega} \}$.
\end{itemize}
Let $T$ be the diagonal matrices subgroup of $G_{2n}$. Define a character on $T\cap S$ by
$$\prescript{\circ}{}{\chi}_{\omega}\left(
    \begin{bmatrix}
    \begin{smallmatrix}
    x_{1} & & &&& \\ 
     & \ddots & &&&\\
    && x_{n}&&&\\
    &&&x_{1} & & \\ 
     &&&& \ddots & \\
    &&&&& x_{n}
    \end{smallmatrix}
    \end{bmatrix}
\right) := \prod_{(i,j)\in \Psi^{um}} \left(\frac{x_i}{x_j}\right),$$ and $\chi_{\omega} := |\prescript{\circ}{}{\chi}_{\omega}|$.
\begin{lemt}\label{Lem: mod char}
    We have $\delta_{P^{\omega} \cap S} = (\delta_P^{\omega} \cdot \chi_{\omega})^{\frac{1}{2}}$ on $T\cap S$. In particular, if $\CO_{\omega}$ is a matching orbit, then we have $\delta_{P^{\omega} \cap S} = (\delta^{\omega}_P)^{\frac{1}{2}}$ on $R_{\omega}$
\end{lemt}

\begin{proof}
    Since the commutator subgroup of $\GL_n(\BK)$ is $\SL_n(\BK)$ and considering the Iwasawa decomposition, the second assertion follows from the first one. 
    \par
    For a Lie group G, denote $\prescript{\circ}{}{\delta}_G(g) := \det(\Ad g|_{\g})$, thus $\delta_G = |\prescript{\circ}{}{\delta}_G|$.
    Denote
    \[
    t = 
    \begin{bmatrix}
    x_{1} & & \\ 
     & \ddots & \\
    && x_{2n}
    \end{bmatrix} 
    \in T \cap S, \text{ where }
    x_i = x_{n+i},\ \forall\ i \in \{1,\dots,n\}.
    \]
    \par
    Denote $\Lambda_1 := \{(i,j)|\ 1 \leq i < j \leq n,\ \omega(i) < \omega(j)\}$, $\Lambda_2 := \Lambda^{ma}_{P,\omega}$,  $\Lambda_3 := \{(i,j)|\ 1 \leq i < j \leq n,\ \omega(i) < \omega(n+j)\}$. Then we have
    $$
    \prescript{\circ}{}{\delta}_{P^{\omega} \cap S}(t) = \prod_{(i,j)\in \Lambda_1} \frac{x_i}{x_j} \cdot \prod_{i\in \Lambda_2} \left(\frac{x_i}{x_{i+1}}\right)^{-1} \cdot \prod_{(i,j)\in \Lambda_3} \frac{x_i}{x_j}.
    $$
    Denote $\Lambda_4 := \{1\leq i \leq 2n | \alpha_i \in \Delta_P\}$, then
    $$
    \prescript{\circ}{}{\delta}^{\omega}_{P} = \prod_{1 \leq i < j \leq 2n} \frac{x_{\omega^{-1}(i)}}{x_{\omega^{-1}(j)}} \cdot \prod_{i\in \Lambda_4} \left(\frac{x_{\omega^{-1}(i)}}{x_{\omega^{-1}(i+1)}}\right)^{-1}.
    $$
    By variable substitution, we obtain    
    $$
    \prescript{\circ}{}{\delta}^{\omega}_{P}\cdot \prescript{\circ}{}{\chi}_{\omega}(t) = \prod_{1 \leq \omega(i) < \omega(j) \leq 2n} \frac{x_i}{x_j} \cdot \prod_{i\in \Lambda_2} \left(\frac{x_i}{x_{i+1}}\right)^{-2}.
    $$
    \par
    Denote $\Lambda_5 := \{(i,j)| 1 \leq i, j \leq 2n, \omega(i) < \omega(j)\}$, and associate a partition to it by
    $$
    \begin{aligned}
    \Lambda_5 =&\quad \ \Lambda_5^{(1)} := \{(i,j)|\ 1 \leq i < j \leq n,\ \omega(i) < \omega(j)\} \\
    &\amalg \Lambda_5^{(2)} := \{(n+b,n+c)|\ 1 \leq b,c \leq n,\ \omega(n+b) < \omega(n+c)\} \\
    &\amalg \Lambda_5^{(3)} := \{(i,n+c)|\ 1 \leq i,c \leq n,\ \omega(i) < \omega(n+c)\}\\
    &\amalg \Lambda_5^{(4)} := \{(n+a,j)|\ 1 \leq a, j \leq n,\ \omega(n+a) < \omega(j)\}.
    \end{aligned}
    $$
    Then by direct calculation,
    $$
    \prod_{(i,j)\in \Lambda_5^{(1)} \amalg \Lambda_5^{(2)}} \frac{x_i}{x_j} = \prod_{(i,j)\in \Lambda_1} \left(\frac{x_i}{x_{j}}\right)^{2}, 
    \quad
    \prod_{(i,j)\in \Lambda_5^{(3)} \amalg \Lambda_5^{(4)}} \frac{x_i}{x_j} = \prod_{(i,j)\in \Lambda_3} \left(\frac{x_i}{x_{j}}\right)^2.
    $$
    Thus $(\prescript{\circ}{}{\delta}_{P^{\omega} \cap S})^2(t) = \prescript{\circ}{}{\delta}^{\omega}_{P} \cdot \prescript{\circ}{}{\chi}_{\omega}(t)$ for $t \in T\cap S$.
\end{proof}

\section{Homology of standard modules}\label{Sec: thmA}
In this section, we investigate the Schwartz homology of the standard module, and derive Theorem \ref{Main thmA} in Subsection \ref{Subsec: thmA}. In Subsection \ref{Subsec: Homological finiteness}, we prove Lemma \ref{Lem: unip orbit vanishing}, \ref{Lem: orbitwise homological finiteness}, which are crucial for our application of Borel's Lemma. In Subsection \ref{Subsec: Homological vanishing}, we prove Lemma \ref{Lem: der um vanish}, which serves as a key step in proving Theorem \ref{Main thmA}.
\par
We fix the following notation in this section. Let $S$ be the Shalika subgroup of $G_{2n}$, and let $P$ be a standard cuspidal parabolic subgroup of $G_{2n}$ corresponding to the partition $(n_1,\dots ,n_r)$, with Levi decomposition $P = L \ltimes N$. We are divided into the following two cases.\par
If $\BK = \BC$, then $n_i = 1$ for $i \in \{1,\dots,r=2n\}$. Take $\pi_i := \chi_{k_i,\lambda_i}$ to be a character of $\BC^{\times}$, where $k_i \in \BZ$, $\lambda_i \in \BC$.\par
If $\BK =\BR$, then $n_i \in \{1,2\}$, for $i \in \{1,\dots,r\}$. For $n_i = 1$, take $\pi_i := \chi_{k_i,\lambda_i}$ to be a character of $\BR^{\times}$, where $k_i \in \{0,1\}$, $\lambda_i \in \BC$. For $n_i = 2$, take $\pi_i := D_{k_i,\lambda_i}$ to be a relative discrete series of $\GL_2(\BR)$, where $k_i \in \BZ_{\geq 1}$, $\lambda_i \in \BC$. \par
In both cases, we ask $\exp(\pi_1) \geq \cdots \geq \exp(\pi_r)$ and denote $\pi := \pi_1 \widehat{\otimes} \cdots \widehat{\otimes} \pi_r$. Let $X:= P\backslash G_{2n}$ be the partial flag manifold and let $\CE$ be the tempered bundle on $X$ associated with the standard module $\pi_1 \dot{\times} \cdots \dot{\times} \pi_r$, i.e. $ \Gamma^{\CS}(X,\CE) \cong \pi_1 \dot{\times} \cdots \dot{\times} \pi_r$.

\subsection{Homological finiteness}\label{Subsec: Homological finiteness}
As explained in Subsection \ref{Subsec: outline}, we need to prove that for each orbit, the corresponding homology group is finite-dimensional.
In this section, we establish Lemma \ref{Lem: unip orbit vanishing}, \ref{Lem: orbitwise homological finiteness}, which will be used in the subsequent sections.
Firstly, we consider the case of $\psi$-vanishing orbits.
\begin{lemt}\label{Lem: unip orbit vanishing}
    For an $\psi$-vanishing $S$-orbit $\CO_{\omega}$, we have
    $$
    \oH^{\CS}_{i}(P^{\omega}\cap S , (\delta_P^{\frac{1}{2}}\pi)^{\omega} \otimes \Sym^k(\CN^{*}_{\omega,\BC}) \otimes \delta^{-1}_{P^{\omega}\cap S} \otimes \xi_{\eta,\psi}^{-1}) = 0, \quad \forall\ i \in \BZ,\ k \in \BZ_{\geq 0}.
    $$
\end{lemt}
\begin{proof}
    Note that $N^{\omega}\cap U$ is a normal subgroup of $P^{\omega}\cap S$. Using the spectral sequence argument, it suffices to prove
    $$
    \oH^{\CS}_{i}(N^{\omega}\cap U , (\delta_P^{\frac{1}{2}}\pi)^{\omega} \otimes \Sym^k(\CN^{*}_{\omega,\BC}) \otimes \delta^{-1}_{P^{\omega}\cap S} \otimes \xi_{\eta,\psi}^{-1}) = 0, \quad \forall\ i \in \BZ,\ k \in \BZ_{\geq 0}.
    $$
    Since $N^{\omega}\cap U$ acts on $\Sym^k(\CN^{*}_{\omega,\BC})$ algebraically, we may take a finite filtration on $\Sym^k(\CN^{*}_{\omega,\BC})$ such that each grading piece $F_j$ admits a trivial action of $N^{\omega}\cap U$. 
    Since $\psi |_{N^{\omega} \cap U} \neq Id $,
    using Lemma \ref{Lem: nilp vanish lem}, we obtain
    $$
    \oH^{\CS}_{i}(N^{\omega}\cap U , (\delta_P^{\frac{1}{2}}\pi)^{\omega} \otimes F_j \otimes \delta^{-1}_{P^{\omega}\cap S} \otimes \xi_{\eta,\psi}^{-1}) = 0, \quad \forall\ i \in \BZ.
    $$
    Thus the lemma follows from the standard long exact sequence argument.
\end{proof}

Denote $T_2(\BR)$ (resp. $N_2(\BR)$) to be the diagonal matrices (resp. the strictly upper triangular matrices) subgroup of $\GL_2(\BR)$, $B_2(\BR):= T_2(\BR) \ltimes N_2(\BR)$.
The following three lemmas concern the homology of the relative discrete series of $\GL_2(\BR)$.

\begin{lemt}\label{Lem: finiteness jacquet module}
Let $D_k\ (k \in \BZ_{\geq1})$ be the discrete series of $\GL_2(\BR)$ as in Subsection \ref{Subsec: LLC}. Then as $T_2(\BR) = \BR^{\times} \times \BR^{\times}$ modules, 
$$\oH^{\CS}_{0}(N_2(\BR), D_k) =(\epsilon|\cdot|^{\frac{k+1}{2}} \boxtimes |\cdot|^{-\frac{k+1}{2}} )\oplus (\epsilon \sgn |\cdot|^{\frac{k+1}{2}} \boxtimes \sgn |\cdot|^{-\frac{k+1}{2}}), $$
where $\epsilon$ is the sign character of $\BR^{\times}$ if $k$ is even and is the trivial character otherwise. Moreover,
$$\oH^{\CS}_{i}(N_2(\BR), D_k) = 0, \quad \forall\ i \in \BZ_{\geq 1}. $$ 
\end{lemt}

\begin{proof}
    Thanks to the comparison theorem for the minimal parabolic case (see \cite[Theorem 1]{hecht1998remark}, \cite[Theorem 5.2]{li2021proof}), we can calculate the homology using the corresponding $(\g,K)$-modules. Then the second equality follows from \cite[Corollary 2.6]{casselman1978restriction} and \cite[Corollary 3.6]{knapp2016cohomological}.
    As to the first equality, it's a direct calculation that $\dim \oH^{\CS}_{0}(N_2(\BR), D_k) = 2$, then the assertion follows from Frobenius reciprocity.
\end{proof}

\begin{lemt}\label{Lem: finiteness whittaker}
    Let $V$ be a relative discrete series of $\GL_2(\BR)$.  Denote by $\psi$ a non-trivial unitary character on $N_2(\BR)$. Then
    \[
    \oH^{\CS}_{i}(N_2(\BR), V \otimes \psi) =
    \left\{
    \begin{array}{ll} 
    \BC, & \text{if } i = 0; \\
    0, & \text{otherwise}.
    \end{array}
    \right.
    \]
\end{lemt}

\begin{proof}
    Recall that we have a standard short exact sequence of representation of $\GL_2(\BR)$,
    $$
    0\rightarrow F \rightarrow I\rightarrow V\rightarrow 0
    $$
    where $I$ is a principal series representation and $F$ is finite-dimensional. Thus the lemma follows from \cite[Lemma 8.5]{Casselman2000Bruhat} and the fact that $\oH^{\CS}_{0}(N_2(\BR), F \otimes \psi) = 0.$
\end{proof}

\begin{lemt}\label{Lem: finiteness pair}
    Let $V_1, V_2$ be relative discrete series of $\GL_2(\BR)$, and let $E$ be a finite-dimensional representation of $\GL_2(\BR)$, then we have,
    $$\dim \oH^{\CS}_{i}(\GL_2(\BR) ,V_1 \widehat{\otimes} V_2 \otimes E) < \infty,\quad \forall\ i \in \BZ.$$
\end{lemt}

\begin{proof}
    Consider the standard exact sequence for $V_j\ (j = 1,2)$ as in Lemma \ref{Lem: finiteness whittaker},
    $$0\rightarrow F_{j}\rightarrow I_{j}\rightarrow V_{j}\rightarrow 0.$$
    Then the lemma follows from the standard arguments of the long exact sequence. 
\end{proof}

Denote $I^{um}_{\omega} := \{i \in \{1,\dots,r\}| \alpha_{a_i} \in \Delta^{um}_{P,\omega}\}$. We shall split $\pi$ into two parts. Let $\pi_{um,\omega} := \widehat{\otimes}_{i\in I^{um}_{\omega}} \pi_i $ and let $\pi_{ma,\omega}$ be the remain part such that $\pi = \pi_{um,\omega} \widehat{\otimes} \pi_{ma,\omega}$. When the context is clear, we will omit the subscript $\omega$ for short.

\begin{lemt}\label{Lem: orbitwise homological finiteness}
    For an $\psi$-unvanishing $S$-orbit $\CO_{\omega}$, we have
    $$
    \dim \oH^{\CS}_{i}(P^{\omega}\cap S , (\delta_P^{\frac{1}{2}}\pi)^{\omega} \otimes \Sym^k(\CN^{*}_{\omega,\BC}) \otimes \delta^{-1}_{P^{\omega}\cap S} \otimes \xi_{\eta,\psi}^{-1}) < \infty, \quad \forall\ i \in \BZ,\ k \in \BZ_{\geq 0}.
    $$
\end{lemt}

\begin{proof}
    For $\BK = \BC$, the lemma is obvious, since the representation occurring in the homology is already finite-dimensional. We consider the case for $\BK = \BR$ in the following.\par
    Recall that in Lemma \ref{Lem: gp decomp} we have decomposition $P^{\omega}\cap S = R_{\omega} \ltimes V_{\omega}$. Following the spectral sequence, it suffices to prove that 
    $$
    \dim \oH^{\CS}_{i}(R_{\omega} , \oH^{\CS}_{j}(V_{\omega} ,(\delta_P^{\frac{1}{2}}\pi)^{\omega} \otimes \Sym^k(\CN^{*}_{\omega,\BC}) \otimes \delta^{-1}_{P^{\omega}\cap S} \otimes \xi_{\eta,\psi}^{-1})) < \infty, \quad \forall \ i,j \in \BZ.
    $$
    Since $V_{\omega}$ acts trivially on $\pi^{\omega}_{ma}$, we have
    $$
    \begin{aligned}
    &\oH^{\CS}_{i}(R_{\omega} , \oH^{\CS}_{j}(V_{\omega} ,(\delta_P^{\frac{1}{2}}\pi)^{\omega} \otimes \Sym^k(\CN^{*}_{\omega,\BC}) \otimes \delta^{-1}_{P^{\omega}\cap S} \otimes \xi_{\eta,\psi}^{-1}))\\
    = &\oH^{\CS}_{i}(R_{\omega}, \oH^{\CS}_{j}(V_{\omega}, \pi_{um}^{\omega}) \widehat{\otimes} \pi_{ma}^{\omega} \otimes \Sym^k(\CN^{*}_{\omega,\BC})
    \otimes (\delta_P^{\frac{1}{2}})^{\omega} \otimes \delta^{-1}_{P^{\omega}\cap S} \otimes \xi_{\eta,\psi}^{-1})
    \end{aligned}
    $$
    Note that $\v_{\omega}$ admits a filtration such that each grading piece isomorphic to $\n_2(\BR)$. Thus following from Lemma \ref{Lem: simple roots alive}, \ref{Lem: finiteness jacquet module}, we know that
    $\dim \oH^{\CS}_{j}(V_{\omega}, \pi_{um}^{\omega}) < \infty$. Note that $R_{\omega}$ is product of copies of $\GL_1(\BR)$, $\GL_2(\BR)$ (whose roots corresponding to $\Phi^{ma}_{P,\omega}$) and $\GL_1(\BR) \ltimes N_2(\BR) \subset B_2(\BR)$ (whose roots corresponding to $\Phi^{wh}_{P,\omega}$). Thus the lemma follows from the \kun formula and Lemma \ref{Lem: finiteness whittaker}, \ref{Lem: finiteness pair}.
\end{proof}

\subsection{Homological vanishing for unmatching orbits and normal derivative}\label{Subsec: Homological vanishing}
In this section, we prove Lemma \ref{Lem: der um vanish}, which is crucial for proving \ref{Main thmA}.
We further introduce some notations in order to provide proof. We fix a $\psi$-unvanishing orbit $\CO{\omega}$. \par

\begin{dfnt}
    For $1\leq i < j \leq n$, we say $i$ is \textbf{weightly related} to $j$, if $(i,j) \in \Lambda^{um}_{P,\omega}$. We say $i$ is \textbf{weightly associated} to $j$, if there is a chain $i = i_0 < i_1 < \cdots <i_l < i_{l+1} = j$, such that $i_k$ is weightly related to $i_{k+1}$, $\forall\ k \in \{0,\dots,l\}$, denoted by $i \rightsquigarrow j$.
\end{dfnt}

\begin{dfnt}
    For $1\leq i < j \leq n$, we say $i$ is \textbf{derivatively related} to $j$, if $E_{j,i}$ or $E_{n+j,i} \in \CN_{\omega}$. We say $i$ is \textbf{derivatively associated} to $j$, if there is a chain $i = i_0 < i_1 < \cdots <i_l < i_{l+1} = j$, such that $i_k$ is derivatively related to $i_{k+1}$, $\forall\ k \in \{0,\dots,l\}$, denoted by $i \rightarrowtail j$.
\end{dfnt}

For a character $\chi_{k,\lambda}$ of $\BK^{\times}$, we say it is \textbf{positive} (resp. \textbf{negative}, \textbf{non-negative}, \textbf{non-positive}), if $\Re \lambda > 0$ (resp. $\Re \lambda < 0$, $\Re \lambda \geq 0$, $\Re \lambda \leq 0$).

\begin{lemt}\label{Lem: non-neg exp}
    For $1 \leq i < j \leq n$.\par
    (i) If $i \rightsquigarrow j$, then 
    $
    \chi^{ex}_{\omega(i)} \cdot \chi^{ex}_{\omega(n+i)} \cdot (\chi^{ex}_{\omega(j)})^{-1} \cdot (\chi^{ex}_{\omega(n+j)})^{-1}
    $
    is non-negative.\par
    (ii) If $i \rightarrowtail j$, then 
    $
    \chi^{ex}_{\omega(i)} \cdot \chi^{ex}_{\omega(n+i)} \cdot (\chi^{ex}_{\omega(j)})^{-1} \cdot (\chi^{ex}_{\omega(n+j)})^{-1}
    $
    is non-negative.
\end{lemt}

\begin{proof}
    It suffices to prove for the case when $i, j$ is related. Thus the second assertion follows from Lemma \ref{Lem: Der associated positive exp}.\par
    As to the first assertion, for $1\leq i < j \leq n$, if $\alpha_{i,j}, \alpha_{n+i,n+j} \in \omega^{-1}(\Delta_P)$, since $\omega \in \Omega$, we obtain $\omega(i) < \omega(j)$, $\omega(n+i) < \omega(n+j)$. If $\alpha_{i,n+j} \in \omega^{-1}(\Delta_P)$, then $\omega(n+j) = \omega(i) +1$. Thus $\omega(j) > \omega(n+j)$, $\omega(i) > \omega(n+i)$ follows from the $\psi$-unvanishing condition, thus $\omega(n+i) <\omega(i) < \omega(n+j) < \omega(j)$.
\end{proof}

Denote $J^{ma}_{\omega} := \{i \in \{1,\dots,n-1\}| \alpha_{i} \in \Psi^{ma}_{P,\omega}\}$. The following lemma is crucial for the proof of Theorem \ref{Main thmA}. 

\begin{lemt}\label{Lem: der um vanish}
If $\CO{\omega}$ is an unmatching orbit or $k > 0$, then
    $$
    \oH^{\CS}_{0}(P^{\omega}\cap S , (\delta_P^{\frac{1}{2}}\pi)^{\omega} \otimes \Sym^k(\CN^{*}_{\omega,\BC}) \otimes \delta^{-1}_{P^{\omega}\cap S} \otimes \xi_{\eta,\psi}^{-1}) = 0.
    $$
\end{lemt}
\begin{proof}
    We first prove the case for $\BK = \BR$. The case for $\BK = \BC$ is similar but simpler, we shall give a comment at last. Now we let $\BK = \BR$.\par
    We introduce the following notation in order to prove the theorem. Recall we have defined a sequence of number $\{a_i,b_i\}$ under Definition \ref{Def: matching orbit} such that $\{a_i\} \amalg \{b_i\} = \{1,\cdots,2n\}$. Define $\chi^{ex}_{a_i} := |\cdot|^{\lambda_i}$, $\chi^{ex}_{b_i} := |\cdot|^{\lambda_{i}}$. If $i \in I^{um}_{\omega}$, then $n_i = 2$, we define $\chi^{wt}_{a_i} := |\cdot|^{\frac{k_i}{2}}$, $\chi^{wt}_{b_i} := |\cdot|^{-\frac{k_i}{2}}$. For other $k \in \{1,\dots,2n\}$, define $\chi^{wt}_{k} := \BC$. Let $\chi_i := \chi^{ex}_{i} \cdot \chi^{wt}_{a_i}$, $i \in \{1,\dots,2n\}$. We shall call $\chi^{ex}_{i}$ (resp. $\chi^{ex}_{i}$) the exponent (resp. weight) part of $\chi_i$.
    \par
    Note that $P^{\omega} \cap S = R_{\omega} \ltimes V_{\omega}$, and $V_{\omega}$ acts algebraically on $\Sym^k(\CN^{*}_{\omega,\BC})$, we may take a finite filtration on $\Sym^k(\CN^{*}_{\omega,\BC})$ such that each grading piece $W$ is an irreducible $R_{\omega}$ module with trivial $V_{\omega}$ action. Thus we only need to show
    \begin{equation}\label{Eq: der un vanishing}
    \oH^{\CS}_{0}(P^{\omega}\cap S , (\delta_P^{\frac{1}{2}}\pi)^{\omega} \otimes W \otimes \delta^{-1}_{P^{\omega}\cap S} \otimes \xi_{\eta,\psi}^{-1}) = 0.
    \end{equation}
    \par
    We first take coinvariance with respect to $V_{\omega}$, then we get 
    $$
    \begin{array}{cl}
    &\oH^{\CS}_{0}(P^{\omega}\cap S , (\delta_P^{\frac{1}{2}}\pi)^{\omega} \otimes W \otimes \delta^{-1}_{P^{\omega}\cap S} \otimes \xi_{\eta,\psi}^{-1}) \\
    = &\oH^{\CS}_{0}(R_{\omega}, \oH^{\CS}_{0}(V_{\omega}, \pi_{um}^{\omega}) \widehat{\otimes} \pi_{ma}^{\omega} \otimes W 
    \otimes (\delta_P^{\frac{1}{2}})^{\omega} \otimes \delta^{-1}_{P^{\omega}\cap S} \otimes \xi_{\eta,\psi}^{-1}).
    \end{array}
    $$
    As in the proof of Lemma \ref{Lem: orbitwise homological finiteness}, $\oH^{\CS}_{0}(V_{\omega}, \pi_{um}^{\omega})$ is tensor product of the Jacquet module of each $\pi_i$ occuring in $\pi_{um}$. Write $E := \oH^{\CS}_{0}(V_{\omega}, \pi_{um}^{\omega})$ for short. 
    Denote by $C(R_{\omega})$ the center of $R_{\omega}$, and let $C(R_{\omega})^{+} \subset C(R_{\omega})$ be the subgroup of the matrices with each entry positive real number.
    Using Lemma \ref{Lem: finiteness jacquet module} and Lemma \ref{Lem: mod char},
    we know that $C(R_{\omega})^{+}$ acts on 
    $$
    E \otimes \pi_{ma}^{\omega} \otimes W 
    \otimes (\delta_P^{\frac{1}{2}})^{\omega} \otimes \delta^{-1}_{P^{\omega}\cap S} \otimes \xi_{\eta,\psi}^{-1}
    $$
    as the character 
    $$
    x = 
    \begin{bmatrix}
    \begin{smallmatrix}
    x_{1} & & & & &\\ 
     & \ddots & & & &\\
    && x_{n}& & &\\
    & & &x_{1} & & \\ 
    & & & & \ddots & \\
    & & & & & x_{n}
    \end{smallmatrix}
    \end{bmatrix} \mapsto
    \alpha(x) \cdot \prod_{i=1}^n \chi_{\omega(i)}(x_i) \cdot \chi_{\omega(n+i)}(x_i)  \cdot \eta^{-1}(x_i),
    $$
    where $\alpha$ is the central character of $W$. 
    In order to prove \eqref{Eq: der un vanishing},
    According to Lemma \ref{Lem: center vanish lem}, in order to prove \eqref{Eq: der un vanishing}, it suffices to show that this character is nontrivial. We use the method of infinite descent.
    \par
    We first deal with the case when $k > 0$.
    Let $\alpha$ be the central character of $W$ with
    \[
    \alpha \left( 
    \begin{bmatrix}
    \begin{smallmatrix}
    x_{1} & & & & &\\ 
     & \ddots & & & &\\
    && x_{n}& & &\\
    & & &x_{1} & & \\ 
    & & & & \ddots & \\
    & & & & & x_{n}
    \end{smallmatrix}
    \end{bmatrix} \right) 
    = \prod_{i=1}^nx_i^{r_i}.
    \]
    Then $\alpha$ is one of the characters of the adjoint action of $C(R_{\omega})^{+}$ on $\Sym^k(\CN^{*}_{\omega,\BC})$.
    Thus we have $\sum r_i = 0,\ r_i \in \BZ$. For $i \in J^{ma}_{\omega}$, let $t_{i} = t_{i+1} := \frac{r_i+r_{i+1}}{2}$. For other $i \in \{1,\dots,n\}$, let $t_{i} := r_{i}$.
    \par
    If the action of $C(R_{\omega})^{+}$ is trivial, then we have 
    \begin{equation}\label{Eq: der vanishing}
    \chi_{\omega(i)}(x) \cdot \chi_{\omega(n+i)}(x) \cdot \eta^{-1}(x) \cdot x^{t_i} = 1, \quad \forall\ x \in \BR^{\times},\ i \in \{1,\dots,n\} . 
    \end{equation}
    Denote $a := \min \{ i \in \{1, \dots, n\} | t_i \neq 0,\ r_i \neq 0\}$. Then $t_a > 0$ since the form of $\CN_{\omega}$. Denote $\Omega_a := \{ i = 1, \dots, n | a \rightarrowtail i \}$, and let $b := \max \{ i \in \Omega_a | t_i \neq 0\}$, then we can easily deduced that $t_b < 0$, also $a < b$.\par
    If either $\chi^{wt}_{\omega(a)} $ or $ \chi^{wt}_{\omega(n+a)}$ is negative, then by definition, there exist $a' \backsim a$ such that both $\chi^{wt}_{\omega(a')} $ and $ \chi^{wt}_{\omega(n+a')}$ are non-negative, with at least one positive, also we have $t_{a'} = 0$. If $\chi^{wt}_{\omega(a)} $ and $ \chi^{wt}_{\omega(n+a)}$ are both non-negative, then we let $a' = a$. We have two possible cases corresponding to these two: (i). $t_{a'} = 0$ and $\chi^{wt}_{\omega(a')} \cdot \chi^{wt}_{\omega(n+a')}$ is positive, (ii). $t_{a'} > 0$ and $\chi^{wt}_{\omega(a')} \cdot \chi^{wt}_{\omega(n+a')}$ is non-negative.
    \par
    Now, if $\chi^{wt}_{\omega(b)} $ and $ \chi^{wt}_{\omega(n+b)}$ are both non-positive, then $\chi_{\omega(a')} \cdot \chi_{\omega(n+a')}\cdot \chi^{-1}_{\omega(b)} \cdot \chi^{-1}_{\omega(n+b)} \cdot x^{t_{a'}-t_b}$ has non-negative exponent part (according to Lemma \ref{Lem: non-neg exp}), positive derivative part and non-negative weight part, contradict with the identity (\ref{Eq: der vanishing}). Thus at least one of $\chi^{wt}_{\omega(b)} $ and $ \chi^{wt}_{\omega(n+b)}$ is positive. Thus there exist $c > b,\ b \rightsquigarrow c$ such that both $\chi^{wt}_{\omega(c)} $ and $ \chi^{wt}_{\omega(n+c)}$ are non-positive, with at least one negative. If $t_c \leq 0$, then $\chi_{\omega(a')} \cdot \chi_{\omega(n+a')}\cdot \chi^{-1}_{\omega(c)} \cdot \chi^{-1}_{\omega(n+c)} \cdot x^{t_{a'}-t_c}$ has non-negative derivative part, positive weight part and non-negative exponent part, contradict with the identity (\ref{Eq: der vanishing}), thus $t_c > 0$. If $r_c \neq 0$, denote $c' = c$. If $r_c = 0$, then $(c-1)\in J^{ma}_{\omega}$ or $c\in J^{ma}_{\omega}$, thus there exist $c' \in \{c+1, c-1\}$ such that $r_{c'} > 0$.  
    In all the case we have $r_{c'} > 0$ and $\Omega_{c'} \neq \varnothing$. Take $d := \max \{ i \in \Omega_{c'} | t_i \neq 0\}$, then $t_d < 0$. If $\chi^{wt}_{\omega(d)} $ and $ \chi^{wt}_{\omega(n+d)}$ are both non-positive, then $\chi_{\omega(a')} \cdot \chi_{\omega(n+a')}\cdot \chi^{-1}_{\omega(d)} \cdot \chi^{-1}_{\omega(n+d)} \cdot x^{r_{a'}-r_d}$ has positive derivative part, non-negative weight part and non-negative exponent part, contradict with the identity (\ref{Eq: der vanishing}). Thus we find $d > c' \geq b$, such that at least one of $\chi^{wt}_{\omega(d)}$ and $ \chi^{wt}_{\omega(n+d)}$ is positive. Since $n$ is finite, using the method of infinite descent, we find the contradiction.\par
    Now we address the case where $k = 0$ and the orbits are unmatching.
    Similarly, we assume that the central character of $C(R_{\omega})^{+}$ is trivial, i.e. we have 
    \begin{equation}\label{Eq: um vanishing}
    \chi_{\omega(i)} \cdot \chi_{\omega(n+i)} \cdot \eta^{-1} = id,\ \forall\ i \in \{1,\dots,n\}.
    \end{equation}
    Take $j := \max \{i\in \{1,\dots,n\}|\text{ at least one of }\chi^{wt}_{\omega(j)} \text{ and }\chi^{wt}_{\omega(n+j)} \text{ is negative}\}$, which is well-defined since $\CO_{\omega}$ is an unmatching orbit. Thus both $\chi^{wt}_{\omega(j)}$ and $\chi^{wt}_{\omega(n+j)}$ are non-positive.    
    Let $a :=\min \{i\in \{1,\dots,n\}| i \rightsquigarrow j\}$, thus both $\chi^{wt}_{\omega(a)} $ and $\chi^{wt}_{\omega(n+a)}$ are non-negative. Thus $\chi_{\omega(a)} \cdot \chi_{\omega(n+a)}\cdot \chi^{-1}_{\omega(j)} \cdot \chi^{-1}_{\omega(n+j)}$ has positive weight part and non-negative exponent part, contradict with the identity (\ref{Eq: um vanishing}).
    \par
    As to the case where $\BK = \BC$, $P$ is just the Borel subgroup, and the standard module is a principal series representation. The main difference lies in the part concerning the complexified conormal bundle. In this case $\CN_{\omega,\BC} \cong \CN_{\omega} \oplus \CN_{\omega}$ as $\BC$-vector spaces, where $z \in \BC^{\times}$ will act by $z$ and $\bar{z}$ separately. However, since $|z| = |\bar{z}|$, the argument as above also works.
\end{proof}

\subsection{Proof of Theorem \ref{Main thmA}}\label{Subsec: thmA}

\begin{thm}\label{Thm: kill unmatching orbits}
The notation is as shown above. We have
$$
\dim \oH^{\CS}_0(S,\Gamma^{\varsigma}(X,\CE)\otimes \xi_{\eta,\psi}^{-1}) \leq \sum_{\omega \text{ matching orbits}} \dim \oH^{\CS}_0(S,\Gamma^{\varsigma}(\CO_{\omega},\CE)\otimes \xi_{\eta,\psi}^{-1}),
$$
where the right hand side is finite.
\end{thm}

\begin{proof}
    Following the $S$-orbit decomposition on $X$ as in Proposition \ref{Prop: orbit decomposition}, $X$ admits a finite decreasing sequence of open submanifolds
    $$U_1 := X \supsetneq U_2 \supsetneq ... \supsetneq U_{f} \supsetneq U_{f+1} := \varnothing,$$
    such that for $i \in \{1, \dots, f \}$, $U_i\backslash U_{i+1}$ is a $S$-orbit in $X$. When $U_i\backslash U_{i+1} = \CO_{\omega}$, we shall also denote $U_{\omega} := U_i$. 
    Consider the short exact sequence
    $$
    0 \to \Gamma^{\varsigma}(U_{i+1}, \CE) \otimes \xi_{\eta,\psi}^{-1} \to \Gamma^{\varsigma}(U_i, \CE) \otimes \xi_{\eta,\psi}^{-1} \to \Gamma^{\varsigma}_{O_{i}}(U_i,\CE) \otimes \xi_{\eta,\psi}^{-1} \to 0,
    $$
    then the associated long exact sequence shows that
    $$
    \dim\ \oH^{\CS}_0(S,\Gamma^{\varsigma}(X,\CE)\otimes \xi_{\eta,\psi}^{-1}) \leq \sum_{\omega \in \Omega} \dim\ \oH^{\CS}_0(S,\Gamma^{\varsigma}_{O_{\omega}}(U_{\omega},\CE)\otimes \xi_{\eta,\psi}^{-1}).
    $$
    According to Borel's lemma, Shapiro's lemma, and Proposition \ref{Prop: homology borel lem}, in order to address the right hand side of the inequality, it suffices to concern with
    $$
    \oH^{\CS}_{i}(P^{\omega}\cap S , (\delta_P^{\frac{1}{2}}\pi)^{\omega} \otimes \Sym^k(\CN^{*}_{\omega,\BC}) \otimes \delta^{-1}_{P^{\omega}\cap S} \otimes \xi_{\eta,\psi}^{-1}),
    $$
    which has been calculated in Lemma \ref{Lem: orbitwise homological finiteness}, \ref{Lem: unip orbit vanishing}, \ref{Lem: der um vanish}.
    By combining these three lemmas, we obtain that 
    \begin{equation*}
    \oH^\CS_{0}(S, \Gamma^{\varsigma}_{O_{\omega}}(U_{\omega},\CE) \otimes \xi_{\eta,\psi}^{-1})
    \cong
    \left\{ 
        \begin{array}{ll}
        \oH^\CS_{0}(S, \Gamma^{\varsigma}(\CO_{\omega},\CE) \otimes \xi_{\eta,\psi}^{-1}), & \textrm{if $\CO_{\omega}$ is a matching orbit};\\
        0, &  \textrm{otherwise},
        \end{array}
    \right.
    \end{equation*}
    and thus the theorem. The finiteness of the right hand side follows from Lemma \ref{Lem: finiteness whittaker}, \ref{Lem: finiteness pair} and the definition of the matching orbits.
\end{proof}

\begin{cort}\label{Cor: Shalika coinv finite} 
    For an irreducible Casselman-Wallach representation $\pi$ of $G_{2n}$, we have
    $$
    \dim\ \oH^{\CS}_0(S,\pi \otimes \xi_{\eta,\psi}^{-1}) < \infty.
    $$
\end{cort}

\begin{proof}
    It suffices to show that $\dim\ \oH^{\CS}_0(S,\sigma \otimes \xi_{\eta,\psi}^{-1}) < \infty$ for each standard module $\sigma$, which follows from Theorem \ref{Thm: kill unmatching orbits}.
\end{proof}

Following Theorem \ref{Thm: kill unmatching orbits}, it is important to study the homology of the Schwartz section over the matching orbits, which will be shown in the Proposition \ref{Prop: symplectic L parameter and homology}.

\begin{prpt}\label{Prop: symplectic L parameter and homology}
    If there exists a matching orbit $\CO_{\omega}$ such that $ \dim \oH^{\CS}_0(S,\Gamma^{\varsigma}(\CO_{\omega},\CE)\otimes \xi_{\eta,\psi}^{-1}) \neq 0$, then the corresponding L-parameter is of $\eta$-symplectic type.
\end{prpt}
\begin{proof}
    Recall we have decomposition $P^{\omega}\cap S = R_{\omega} \ltimes V_{\omega}$. When $\CO_{\omega}$ is a matching orbit, we know that $V_{\omega}$ acts trivially on the representation, thus we have
    $$
    \begin{aligned}
    \oH^{\CS}_0(S,\Gamma^{\varsigma}(\CO_{\omega},\CE)\otimes \xi_{\eta,\psi}^{-1}) & = \oH^{\CS}_{0}(P^{\omega}\cap S , (\delta_P^{\frac{1}{2}}\pi)^{\omega} \otimes \delta^{-1}_{P^{\omega}\cap S} \otimes \xi_{\eta,\psi}^{-1})\\
    & = \oH^{\CS}_{0}(R_{\omega} , (\delta_P^{\frac{1}{2}}\pi)^{\omega} \otimes \delta^{-1}_{P^{\omega}\cap S} \otimes \xi_{\eta,\psi}^{-1})\\
    & = \oH^{\CS}_{0}(R_{\omega} , \pi^{\omega} \otimes \xi_{\eta,\psi}^{-1})
    \end{aligned},
    $$
    where the first equality follows from Shapiro's lemma and the third follows from Lemma \ref{Lem: mod char}. Therefore, as discussed in the proof of Lemma \ref{Lem: orbitwise homological finiteness}, the theorem follows from \kun formula, Lemma \ref{Lem: finiteness whittaker}, \ref{Lem: finiteness pair} and \cite[Lemma 3.9]{suzuki2023epsilon}.
    In fact, recall that we have defined $s_{\omega}$ after Definition \ref{Def: matching orbit}. In this case we have $\pi_i \cong \pi_{s_{\omega}(i)}^{\vee} \cdot \eta$.
\end{proof}

\begin{proof}[Proof of Theorem \ref{Main thmA}]
    Let $\tau$ be an irreducible Casselman-Wallach representation of $G_{2n}$.
    If $\tau$ admits an $(\eta, \psi)$-twisted Shalika functional, then so is its standard module $\pi_1 \dot{\times} \cdots \dot{\times} \pi_r$. 
    Let $X$ be the corresponding partial flag manifold, and let $\CE$ be the corresponding tempered vector bundle on $X$, such that $ \Gamma^{\CS}(X,\CE) \cong \pi_1 \dot{\times} \cdots \dot{\times} \pi_r$.
    Thus $\dim\ \oH^{\CS}_0(S,\Gamma^{\varsigma}(X,\CE)\otimes \xi_{\eta,\psi}^{-1}) \neq 0$. Then Theorem \ref{Main thmA} follows from Theorem \ref{Thm: kill unmatching orbits} and Proposition \ref{Prop: symplectic L parameter and homology}. 
\end{proof}

\section{Proof of Theorem \ref{Main thmB}}\label{Sec: thmB}
We first recall the following theorem, which constructs twisted Shalika functionals for parabolic induced representations.
\begin{thm}[{\cite[Theorem 2.1]{chen2020archimedean}}]\label{Thm: shalika parabolic}
    For two even positive integers $n_1$ and $n_2$, take two Casselman-Wallach representations $\pi_1$ and $\pi_2$ of $\GL_{n_1}(\BK)$ and $\GL_{n_2}(\BK)$, respectively, and assume that both $\pi_1$ and $\pi_2$ have non-zero $(\eta, \psi)$-twisted Shalika periods. Then the normalized parabolic induction $\pi_1 \dot{\times} \pi_2$ also has a non-zero $(\eta, \psi)$-twisted Shalika periods.
\end{thm}

Recall in Lemma \ref{Lem: symplectic parameter}, if an L-parameter is of $\eta$-symplectic type, then it has form 
$$\sum_{i} \phi_i + \sum_j (\phi_j + \phi_j^{\vee} \cdot \eta).$$
Note that for the $\GL_2(\BK)$ case, an irreducible representation $\pi$ has an $(\eta, \psi)$-Shalika period if and only if it has a Whittaker period and the central character of $\pi$ equals to $\eta$. Thus, according to Theorem \ref{Thm: shalika parabolic}, we can reduce the theorem to the following case in $\GL_4(\mathbb{R})$.

For a character $\eta$ of $\BK^{\times}$, we write $\eta_{\GL_m(\BK)} := \eta \circ \det_{\GL_m(\BK)}$ for the associated character of $\GL_m(\BK)$. Denote by $S_{2n}$ the Shalika subgroup of $G_{2n}$.

\begin{thm}\label{Thm: shalika GL4}
    Let $P \subset \GL_4(\BR)$ be the standard parabolic subgroup corresponding to the partition $(2,2)$. Let 
    $\eta : \BR^{\times} \to \BC^{\times},\ t \mapsto |t|^{z_0} (\sgn t)^{m_0},\  z_0 \in \BC^{\times},\ m_0 \in \{0,1\}$
    be a character of $\BR^{\times}$. Assume that $D_{k,\lambda}\ (k \in \BZ_{\geq 1},\ \lambda \in \BC)$ doesn't have a non-zero $(\eta,\psi)$-twisted Shalika peroids (thus so is $D_{k, z_0 - \lambda}$), and $D_{k,\lambda} \dot{\times} D_{k, z_0 - \lambda}$ is irreducible, then $D_{k,\lambda} \dot{\times} D_{k, z_0 - \lambda}$ has a non-zero $(\eta,\psi)$-twisted Shalika peroids. Moreover, we have
    $$
    \oH^{\CS}_{i}(S_4 ,D_{k,\lambda} \dot{\times} D_{k, z_0 - \lambda} \otimes \xi^{-1}_{\eta,\psi}) \cong \oH^{\CS}_{i}(\GL_2(\BR) ,D_{k,\lambda} \widehat{\otimes} D_{k, z_0 - \lambda} \otimes \eta_{\GL_2(\BR)}^{-1}),\quad i\in \BZ,
    $$
    and
    $$
    \Hom_{S_4}(D_{k,\lambda} \dot{\times} D_{k, z_0 - \lambda}, \xi_{\eta,\psi}) \cong \Hom_{\GL_2(\BR)}(D_{k,\lambda} \widehat{\otimes} D_{k, z_0 - \lambda}, \eta_{\GL_2(\BR)}).
    $$
\end{thm}

\begin{proof}
    According to Lemma \ref{Lem: finiteness pair}, the second equality follows from the first one (for $i=0$) by taking continuous dual. In the following, we aim to prove the first equality.
    \par
    A quick calculation shows that there are only 4 $S_4$-orbits on $X := P\backslash \GL_4(\BR)$, represented by $\{ \sigma_1 = (1,2,3,4), \sigma_2 = (1,3,2,4), \sigma_3 = (3,1,2,4), \sigma_4 = (3,4,1,2)\}$, where $\CO_{\sigma_4}$ is the unique open $S_4$-orbit. One can easily verify that $\CO_{\sigma_1}$ and $\CO_{\sigma_3}$ are $\psi$-vanishing, $\CO_{\sigma_2}$ and $\CO_{\sigma_4}$ are matching orbits. Note that $P^{\sigma_4} \cap S_4 \cong \GL_2(\BR)$. 
    Denote by $\CE$ the tempered bundle on $X$ associated  with $D_{k,\lambda} \dot{\times} D_{k, z_0 - \lambda}$. Following the Shapiro's lemma and Lemma \ref{Lem: mod char}, we know that 
    $$
    \oH^{\CS}_{i}(S_4 ,\Gamma^{\varsigma}(\CO_{\sigma_4}, \CE) \otimes \xi^{-1}_{\eta,\psi}) \cong \oH^{\CS}_{i}(\GL_2(\BR) ,D_{k,\lambda} \widehat{\otimes} D_{k, z_0 - \lambda} \otimes \eta_{\GL_2(\BR)}^{-1}),\quad i\in \BZ.
    $$
    \par
    Denote $\pi_{0} := D_{k,\lambda} \widehat{\otimes} D_{k, z_0 - \lambda}$. Using Borel's lemma, Shapiro's lemma, and Proposition \ref{Prop: homology borel lem}, in order to prove the theorem, it suffices to prove that
    \begin{equation}\label{Eq: ThmB}
    \oH^{\CS}_{i}(P^{\omega}\cap S_4 , (\delta_P^{\frac{1}{2}}\pi_{0})^{\omega} \otimes \Sym^l(\CN^{*}_{\omega,\BC}) \otimes \delta^{-1}_{P^{\omega}\cap S_4} \otimes \xi_{\eta,\psi}^{-1}) = 0, \quad \forall\ i,\ l \in \BZ_{\geq 0},
    \end{equation}
    where $\omega \in \{\sigma_1, \sigma_2, \sigma_3\}$.
    Since $\CO_{\sigma_1}$ and $\CO_{\sigma_3}$ are $\psi$-vanishing, (\ref{Eq: ThmB}) in these two cases follows from Lemma \ref{Lem: unip orbit vanishing}.
    \par
    Since $\sigma_2$ is matching, the terms of modular characters can be canceled. Thus it suffices to prove
    $$
    \oH^{\CS}_{i}(P^{\sigma_2}\cap S_4 , \pi_{0}^{\sigma_2} \otimes \Sym^l(\CN^{*}_{\sigma_2, \BC}) \otimes \xi_{\eta,\psi}^{-1}) = 0, \quad \forall\ i,\ l \in \BZ_{\geq 0}.
    $$
    We first prove the case when $i=l=0$. One can easily calculate that $R_{\sigma_2} \cong S_2 \times S_2$ and $V_{\sigma_2} = N^{\sigma_2}\cap S_4$, where $N$ is the unipotent radical of $P$. Thus 
    $$
    \oH^{\CS}_{0}(P^{\sigma_2}\cap S_4 , \pi_{0}^{\sigma_2} \otimes \xi_{\eta,\psi}^{-1})  = \oH^{\CS}_{0}(S_2 \times S_2 , (D_{k,\lambda}\otimes \xi_{\eta,\psi}^{-1}) \widehat{\otimes} (D_{k, z_0 - \lambda}\otimes \xi_{\eta,\psi}^{-1})) = 0,
    $$
    where the second equality follows from \kun formula and the condition that $D_{k,\lambda}$ doesn't admit a $(\eta,\psi)$-twisted Shalika peroids.
    \par
    Now we consider the case when $i>0$ or $l>0$. A center element in $P^{\sigma_2}\cap S_4$ has form 
    \[ 
    \begin{bmatrix}
    a & & &\\ 
    & d & &\\
    & & a & \\
    & &  & d
    \end{bmatrix},
    \]
    it acts on $\oH^{\CS}_{i}(V_{\sigma_2} , \pi_{0}^{\sigma_2} \otimes \Sym^{l}(\CN^{*}_{\sigma_2, \BC}) \otimes \xi_{\eta,\psi}^{-1})$ by 
    $$
    |a|^{2\lambda - z_0 + l +i} |d|^{z_0 - 2\lambda - l -i} (\sgn(ad))^{k+1+m_0+l+i}.
    $$
    If $\Re(2\lambda - z_0) \geq 0$ then $2\lambda - z_0 + l +i \neq 0$. For the case $\Re(2\lambda - z_0) \leq 0$, \cite[Theorem 10b]{Moeglin1997} asserts that $|2\lambda - z_0| \notin \{1,2,3,\dots\}$, thus $2\lambda - z_0 + l +i \neq 0$. Following Lem \ref{Lem: center vanish lem}, we have
    $$
    \oH^{\CS}_{j}(R_{\sigma_2} , \oH^{\CS}_{i}(V_{\sigma_2}, \pi_{0}^{\sigma_2} \otimes \Sym^k(\CN^{*}_{\sigma_2, \BC}) \otimes \xi_{\eta,\psi}^{-1})) = 0.
    $$
    Thus following the spectral sequence argument,
    $$
    \oH^{\CS}_{i}(P^{\sigma_2}\cap S_4 , \pi_{0}^{\sigma_2} \otimes \Sym^l(\CN^{*}_{\sigma_2, \BC}) \otimes \xi_{\eta,\psi}^{-1}) = 0, \quad \forall\ i,\ l \in \BZ_{\geq 0}.
    $$
\end{proof}

\begin{rremark}
    The above proof also applies to the case of the limit of relative discrete series.
\end{rremark}

\begin{proof}[Proof of Theorem \ref{Main thmB}]
    We first consider the case of $\GL_2(\BK)$. An irreducible representation $\pi$ of $\GL_2(\BK)$ has an $(\eta, \psi)$-Shalika period if and only if it has a Whittaker period and the central character of $\pi$ is $\eta$. For $\BK = \BC$, the $\eta$-symplectic L-parameter must has form $\chi + \chi^{\vee} \cdot \eta$. Assuming it is generic, then the corresponding irreducible principal series representation has central character $\eta$. 
    For $\BK = \BR$, there are two cases. In the irreducible principal series case, it coincides with that for $\GL_2(\BC)$. For the irreducible 2-dimensional representation $\sigma_{k,\lambda}$ of $W_{\BR}$, note that $\GL_2(\BC) = \GSp_2(\BC)$, thus the similitude character of $\sigma_{k,\lambda}$ is 
    $\det (\sigma_{k,\lambda}) = \sgn^{k+1} |\cdot|^{2\lambda}$, which is just the central character of $D_{k,\lambda}$.
    \par
    We now turn to the general case of $\GL_{2n}(\BK)$. 
    For $\BK = \BC$, the L-parameter of $\eta$-symplectic type has form
    $$\sum_i (\chi_i + \chi_i^{\vee} \cdot \eta),$$
    thus Theorem \ref{Main thmB} follows from the case of $\GL_2(\BC)$ and Theorem \ref{Thm: shalika parabolic}.
    \par
    For $\BK = \BR$, the L-parameter of $\eta$-symplectic type has form
    $$\sum_{i} \sigma_{k_i, \lambda_i} + \sum_j (\phi_j + \phi_j^{\vee} \cdot \eta),$$
    thus Theorem \ref{Main thmB} follows from the case of $\GL_2(\BR)$, Theorem \ref{Thm: shalika GL4} and Theorem \ref{Thm: shalika parabolic}.
\end{proof}

\section{Theta correspondence and linear periods}\label{Sec: linear model}
As proved in \cite{Gan2019}, Shalika periods are related to linear periods under theta correspondence of the dual pair $(\GL_{2n}(\BK), \GL_{2n}(\BK))$. Denote by $\Theta(\pi)$ the full theta lift of a representation $\pi$.

\begin{thm}[{\cite[Theorem 3.1]{Gan2019}}]\label{Thm: Gan theta shalika linear }
    Let $\BK$ be a non-archimedean local field of characteristic not 2. For any $\pi \in \Irr(\GL(W))$ and $\sigma \in \Irr(G_B)$, one has
    $$
    \Hom_{G_B\ltimes N(V_1)} (\Theta(\pi), \sigma \boxtimes \psi_B) \cong \Hom_{\GL(W_1)\times \GL(W_2)}(\pi^{\vee}, \sigma \boxtimes \BC)
    $$
    where we have regarded $\sigma$ naturally as a representation of $\GL(W_1)$. $($All the notations here will be explained in the following$)$.
\end{thm}

In this section, we prove Theorem \ref{Thm: archimedean theta shalika linear}, which is an analogous result over the archimedean local field. The following Lemma \ref{Lem: AG homology} is essential to our proof. 

\begin{lemt}[{\cite[Lemma 6.2.2]{aizenbud2015twisted}}]\label{Lem: AG homology}
    Let X be a Nash manifold and let V be a real vector space. Let $\phi: X\to V^{\ast}$ be a Nash map. Suppose $0 \in V^{\ast}$ is a regular value of $\phi$. It gives a map $\chi: V \to \CT(X)$ given by $\chi(v)(x) = \theta(\phi(x)(v))$ (where $\CT(X)$ denotes the space of tempered functions on $X$). This gives an action of V on $S(X)$ by $\pi(v)(f) := \chi(v) \cdot f$. Then:
    \begin{enumerate}[label=(\roman*)]
        \item $\oH_i(\v,S(X)) = 0$ for $i > 0$.
        \item Let $X_0 := \phi^{-1}(0)$. Note that it is smooth. Let r denote the restriction map $r: S(X) \to S(X_0)$. Then r gives an isomorphism $\oH_0(\v,S(X)) \xrightarrow{\sim} S(X_0)$.
    \end{enumerate}
\end{lemt}

We introduce the following notations for the Schr\"{o}dinger model as in \cite[Section 3]{Gan2019}. Let $\BK$ be an archimedean local field. Let $V$ and $W$ be $2n$-dimensional vector spaces over $\BK$. Write 
$$ V = V_1 + V_2 \text{ with } \dim V_i = n.$$ 
Denote by $P(V_1) = (\GL(V_1) \times \GL(V_2)) \ltimes N(V_1)$ the parabolic subgroup of $\GL(V)$ stabilizing $V_1$, where $N(V_1) \cong \Hom(V_2, V_1)$. We work on the Schr\"{o}dinger model of the Weil representation 
$$\Omega \cong \CS((W\otimes V_1^{*}) \times (W\otimes V_2^{*})).$$
The action of $P(V_1) \times \GL(W)$ is given by
    $$
    \left\{
    \begin{aligned}
        &(h \cdot f)(T,X) = f(h^{-1} \circ T, X \circ h) \quad\quad\quad\quad\quad\quad\quad\quad\quad\quad\ \ \text{for}\ h\in \GL(W);\\
        &((g_1,g_2) \cdot f)(T,X) = (\det(g_1)/\det(g_2))^{n} \cdot f(T \circ g_1, g_2^{-1} \circ X)  \\
        &\quad \quad \quad \quad \quad \quad \quad \quad \quad\quad\quad\quad\quad\quad\quad\quad\quad\quad\text{for}\ (g_1,g_2)\in \GL(V_1) \times \GL(V_2);\\
        &(n(A) \cdot f)(T,X) = \psi(\Tr_{V_2}(XTA)) \cdot f(T, X ) \quad \quad \quad \quad \text{for}\ A \in \Hom(V_2, V_1).\\
    \end{aligned}
    \right.
    $$
Given $B \in \mathrm{Isom}(V_1,V_2)$, define the character of $N(V_1)$ by $$\psi_B : n(A) \mapsto \psi(\Tr_{V_2}(BA)).$$
The stabilizer of $\psi_B$ in $\GL(V_1) \times \GL(V_2)$ is the diagonally embedded subgroup 
$$
G_B := \{(g,BgB^{-1})|\ g\in \GL(V_1)\} \subset \GL(V_1) \times \GL(V_2).
$$
\par
The following lemma describes the coinvariance space of the Weil representation with respect to $(N(V_1), \psi_B)$.
\begin{lemt}\label{Lem: homology of weil rep}
    \[
    \oH^{\CS}_i(N(V_1), \Omega \otimes \psi_B^{-1}) =
    \left\{
    \begin{array}{ll} 
    S(\CO_B) & \text{if } i = 0; \\
    0 & \text{otherwise},
    \end{array}
    \right.
    \]
    where $$\CO_B := \{(T,X) \in \Hom(V_1, W) \times \Hom(W, V_2) |\ XT=B\}.$$
\end{lemt}

\begin{proof}
    Note that $N(V_1)$ is an abelian group.
    Denote by $\n(V_1)$ the Lie algebra of $N(V_1)$.
    Consider the following map
    \[
    \begin{array}{rccc}
    F: &\Hom(V_1, W) \times \Hom(W, V_2)&\rightarrow &\n^{*}(V_1)\\
    &(T,X)&\mapsto &\left(A\mapsto \Tr((XT-B)A) \right).\\
    \end{array}
    \]
    Note that  $F^{-1}(0) = \CO_B$, the lemma follows from Lemma \ref{Lem: AG homology} once we verify that $0$ is a regular value of $F$. Take $(T_0, X_0) \in \CO_B$, then 
    $$
    \frac{\partial F}{\partial (T_1, X_1)} (T_0, X_0) = (A \mapsto \Tr((X_1 T_0 + x_0 T_1)A) ).
    $$
    We may fix a basis of the above vector space so that the linear map can be expressed by matrices. 
    Since     $T_0$ is a $n \times 2n$ matrix with rank $n$, there exist $X_{i,j} \in \Hom(W,V_2)$ such that $X_{i,j}T_0 = E_{i,j}$, $\forall\ 1\leq i,j\leq n$, where $E_{i,j}$ is the elementary matrix of size $n \times n$. Then 
    $$
    \frac{\partial F}{\partial (0, X_{i,j})} (T_0, X_0) = (A \mapsto a_{j,i} ).
    $$
    Thus $0$ is a regular value of $F$.
\end{proof}

Note that $G_B \times \GL(W)$ acts transitively on $\CO_B$. Take $(T_0, X_0) \in \CO_B$, so that 
$$
W = \Im(T_0) \oplus \Ker(X_0) =: W_1 \oplus W_2,
$$
Then the stabilizer of $(T_0, X_0)$ in $G_B \times \GL(W)$ is 
$$
H_B = \{((g,Bgb^{-1}),(T_0gT_0^{-1}|_{W_1}, h))|\ g \in \GL(V_1),\ h \in \GL(W_2)\}. 
$$
\par
As another ingredient of the proof of Theorem \ref{Thm: archimedean theta shalika linear}, the full theta lift of generic representation is studied in \cite{fang2018godement}.
\begin{lemt}[{\cite[Theorem 1.7]{fang2018godement}}]\label{Lem: theta generic}
    Let $\BK$ be a local field. Assume $\sigma$ is generic representations of $\GL_n(\BK)$, then $\Theta(\sigma) \cong \sigma^{\vee}$.
\end{lemt}

Now we are ready to prove the main theorem in this section.

\begin{thm}\label{Thm: archimedean theta shalika linear}
     The notation is as shown above. Let $\BK$ be an archimedean local field. For any $\pi \in \Irr(\GL(W))$ and $\chi$ a character of $G_B$, one has
    $$
    \Hom_{G_B\ltimes N(V_1)} (\Theta(\pi), \chi \boxtimes \psi_B) \cong \Hom_{\GL(W_1)\times \GL(W_2)}(\pi^{\vee}, \chi \boxtimes \BC)
    $$
    where we regard $\chi$ naturally as a character of $\GL(W_1)$. In particular, for generic $\pi$ we have
    $$
    \Hom_{G_B\ltimes N(V_1)} (\pi, \chi \boxtimes \psi) \cong \Hom_{\GL(W_1)\times \GL(W_2)}(\pi, \chi \boxtimes \BC)
    $$
\end{thm}

\begin{proof}
    The second assertion follows from the first one and Lemma \ref{Lem: theta generic}.
    As to the first equality, the proof is similar to the original one in \cite{Gan2019}, We sketch the proof for the convenience of the readers.
    Note that
    $$
    \begin{aligned}
    &\oH^{\CS}_0(G_B \ltimes N(V_1),\Theta(\pi) \otimes (\chi^{-1} \boxtimes \psi_B^{-1})) \\
    =&  \oH^{\CS}_0(G_B \ltimes N(V_1),\oH^{\CS}_0(\GL(W), \Omega\ \widehat{\otimes} \pi^{\vee}) \otimes (\chi^{-1} \boxtimes \psi_B^{-1}))\\
     =& \oH^{\CS}_0(G_B \times \GL(W),\oH^{\CS}_0(N(V_1), \Omega \otimes \psi_B^{-1}) \otimes \chi^{-1}\ \widehat{\otimes} \pi^{\vee})\\
     =& \oH^{\CS}_0(G_B \times \GL(W),(\ind^{G_B \times \GL(W)}_{H_B} \BC)\ \otimes \chi^{-1}\ \widehat{\otimes} \pi^{\vee})\\
     =& \oH^{\CS}_0(\GL(W_1) \times \GL(W_2), \pi^{\vee} \otimes (\chi^{-1} \boxtimes \BC) )
    \end{aligned}
    $$
    where the third equality follows from Lemma \ref{Lem: homology of weil rep} and the discussion after it, and the fourth equality follows from Shapiro's lemma. The first term occurring in the equality is finite-dimensional according to Corollary \ref{Cor: Shalika coinv finite}, thus the theorem follows by taking continuous dual.
\end{proof}

\section{Restriction to \texorpdfstring{$\GL_{2n}^{+}(\BR)$}{TEXT}}\label{Sec: restrict to GL+}
In this section, we prove Theorem \ref{thmE}.
We fix the following notations in this section. For $m \in \BZ_{\geq 1}$, let $G_m := \GL_m(\BR)$ be the real general linear group of rank $m$, and let $G_m^{+} := \GL_m^{+}(\BR)$ be the identity component of $G_m$. For $a \in \BR^{\times}$, let 
$$\psi_a : \BR \to \BC^{\times},\ x \mapsto \exp(2\pi ax\sqrt{-1}).$$
\par
Let $K_{2n} := \RO_{2n}(\BR) \supset K^{+}_{2n} := \SO_{2n}(\BR)$. We first review the classification of the irreducible representations of $K_{2n}$ and $K_{2n}^{+}$. 
\par
Fix the identification
    \[
    \begin{array}{ccc}
    \SO_2(\BR)&\xrightarrow{\sim} &\RU(1)\\
    \begin{bmatrix}
    \cos \theta & \sin \theta \\ 
    -\sin \theta& \cos \theta 
    \end{bmatrix}&\mapsto &e^{i\theta}.\\
    \end{array}
    \]
According to the highest weight theory, the irreducible representations of $\SO_{2n}(\BR)$ are parametrized by $(a_1,\dots,a_n) \in \BZ^{n}$ such that
$$
a_1 \geq \cdots \geq a_{n-1} \geq |a_n|.
$$
Let $\sigma$ be the irreducible representation of $\SO_{2n}(\BR)$ with highest weight $(a_1,\dots,a_n)$. When $a_n = 0$, there are exactly two non-isomorphic irreducible representations $\tau_1,\tau_2$ of $\RO_{2n}(\BR)$, such that 
$$\tau_i |_{\RO_{2n}(\BR)} \cong \sigma, \quad i = 1,\ 2.$$
When $a_n > 0$, denote by $\sigma'$ the representation corresponding to $(a_1,\dots,-a_n)$, then there is a unique irreducible representation $\tau$ of $\RO_{2n}(\BR)$, such that $\tau |_{\RO_{2n}(\BR)} \cong \sigma \oplus \sigma'$. These give a parametrization of the representation of $\RO_{2n}(\BR)$. When the second case happens, we denote $\tau^{+} := \sigma$ and $\tau^{-} := \sigma'$, thus 
$$\tau |_{\RO_{2n}(\BR)} \cong \tau^{+} \oplus \tau^{-}.$$
\par
Let $\pi$ be an irreducible representation of $G_{2n}$. According to Clifford theory, $\pi |_{G_{2n}^{+}}$ is either an irreducible representation of $G_{2n}^{+}$, or it is reducible and isomorphic to a direct sum of two non-isomorphic irreducible representations of $G_{2n}^{+}$. The following lemma describes the condition for the second case.

\begin{lemt}
    The notation is as shown in Subsection \ref{Subsec: LLC}. Let $\pi$ be an irreducible representation of $G_{2n}$ with its L-parameter 
    $$
    \phi_{\pi} = \sum_i \chi_{0,t_i} + \sum_j \chi_{1,s_j} + \sum_k \sigma_{l,r_k}, \quad t_i, s_j, r_k \in \BC,
    $$
    then $\pi|_{G^{+}_{2n}}$ is reducible if and only if $\chi_{0,t}$ and $\chi_{1,t}$ appear in pairs.
\end{lemt}

\begin{proof}
    Following the Clifford theory, $\pi|_{G^{+}_{2n}}$ is reducible if and only if $\pi \cong \pi \otimes \sgn$. Then the lemma follows from the local Langlands correspondence.
\end{proof}

Denote the limit of relative discrete series of $\GL_2(\BR)$ by 
$$D_{0,\lambda} := |\cdot|^{\lambda} \dot{\times} \sgn|\cdot|^{\lambda},\ \lambda \in \BC.$$
Using induction by stage, the above lemma says that, when $\pi|_{G^{+}_{2n}}$ is reducible, the corresponding standard module has form 
    $$
    D_{k_1, \lambda_1} \dot{\times} \cdots \dot{\times} D_{k_n, \lambda_n},\ k_i \in \BZ_{\geq 0},\ \lambda_i \in \BC.
    $$
Note that the minimal $K_{2n}$-type of this standard module (thus the irreducible quotient) has extremal weight $(k_1+1, \dots, k_n+1)$. Denote by $\tau$ the minimal $K_{2n}$-type of $\pi$. Following the classification of the representations of $K_{2n}$, we have $\tau |_{\SO_{2n}} = \tau^{+} \oplus \tau^{-}$. We shall denote $$\pi|_{G^{+}_{2n}} = \pi^{+} \oplus \pi^{-}$$ 
such that $\tau^{+} \subset \pi^{+}$, $\tau^{-} \subset \pi^{-}$. Note that $\pi^{+}$ and $\pi^{-}$ can be gotten from each other by twisting a matrix $g$ with $\det g = -1$. 
\par
Denote by $S_{2n}$ the Shalika subgroup of $G_{2n}$.
Note that $S_{2n} \subset G^{+}_{2n}$. For an irreducible representation $\pi$ of $G_{2n}$ such that $\pi|_{G^{+}_{2n}} = \pi^{+} \oplus \pi^{-}$, we have the following equation,
    $$
    \Hom_{S_{2n}}(\pi, \xi_{\eta, \psi_a}) = \Hom_{S_{2n}}(\pi^+, \xi_{\eta, \psi_a}) \oplus \Hom_{S_{2n}}(\pi^-, \xi_{\eta, \psi_a}),
    $$
where the left hand side of the equation is at most one-dimensional. 
Assume further that $\pi$ has a non-zero $(\eta, \psi_a)$-twisted Shalika period $\mu$, we define
    \begin{equation}\label{Def: epsilon pi}
    \epsilon_{\pi} :=
    \left\{
    \begin{array}{rl} 
     1, & \text{if } \mu|_{\pi^{+}} \neq 0; \\
    -1, & \text{if } \mu|_{\pi^{-}} \neq 0.
    \end{array}
    \right.
    \end{equation}
\par
In order to calculate $\epsilon_{\pi}$, we first study the behavior of the parabolic induced representations of $G_{2n}$ when restricted to $G_{2n}^{+}$. Let $P$ be a standard parabolic subgroup of $\GL_{2n}(\BR)$ corresponding to the even partition $(n_1,\dots,n_r)$ (i.e. all $n_i$ are even), with Levi decomposition $P = L\ltimes N_P$. Let $P^{+} := P \cap \GL^{+}_{2n}(\BR)$, $L^{+} := L \cap \GL^{+}_{2n}(\BR)$, and let $P^{\circ}$ (resp. $L^{\circ}$) be the identity component of $P$ (resp. $L$), then we have $P^{+} = L^{+} \ltimes N_P$, $P^{\circ} = L^{\circ} \ltimes N_P$. Note that $L^{\circ} = G_{n_1}^{+} \times \cdots \times G_{n_r}^{+}$. 

\begin{lemt}
    Let $P$ be as above, and let $\pi_i$ be a irreducible representation of $G_{n_i}$ such that $\pi_i |_{G_{n_i}^{+}} = \pi_i^{+} \oplus \pi_i^{-}$, then we have
    $$
    (\Ind^{G_{2n}}_P \pi_1 \widehat{\otimes} \cdots \widehat{\otimes} \pi_r) |_{G^{+}_{2n}} \cong (\Ind^{G^{+}_{2n}}_{P^{\circ}} \pi_1^{+} \widehat{\otimes} \cdots \widehat{\otimes} \pi_{r-1}^{+} \widehat{\otimes} \pi_r^{+} )
    \oplus (\Ind^{G^{+}_{2n}}_{P^{\circ}} \pi_1^{+} \widehat{\otimes} \cdots \widehat{\otimes} \pi_{r-1}^{+} \widehat{\otimes} \pi_r^{-}).
    $$
\end{lemt}

\begin{proof}
    Note that $\pi_i \cong \Ind^{G_{n_i}}_{G^{+}_{n_i}} \pi^{+}$, thus 
    $\pi_1 \widehat{\otimes} \cdots \widehat{\otimes} \pi_r \cong \Ind^{L}_{L^{\circ}} \pi_1^{+} \widehat{\otimes} \cdots \widehat{\otimes} \pi_{r-1}^{+} \widehat{\otimes} \pi_r^{+}.
    $
    Denote 
    $
    g' := 
    \begin{bmatrix}
    \begin{smallmatrix}
    1 & & &\\ 
    & \ddots & &\\
    & & 1 & \\
    & &  & -1
    \end{smallmatrix}
    \end{bmatrix}
    \in L \subset G_{2n},
    $
    then 
    $
     L^{\circ} \backslash L / L^{+} = \{ 1, g'\}.
    $
    Thus 
    $$
    (\pi_1 \widehat{\otimes} \cdots \widehat{\otimes} \pi_r) |_{L^{+}} \cong (\Ind^{L^{+}}_{L^{\circ}} \pi_1^{+} \widehat{\otimes} \cdots \widehat{\otimes} \pi_{r-1}^{+} \widehat{\otimes} \pi_r^{+})
    \oplus (\Ind^{L^{+}}_{L^{\circ}} \pi_1^{+} \widehat{\otimes} \cdots \widehat{\otimes} \pi_{r-1}^{+} \widehat{\otimes} \pi_r^{-}).
    $$
    Since $P \backslash G_{2n} / G_{2n}^{+} = \{ 1\}$, 
    we have
    $$
    \begin{aligned}
    \pi |_{G^{+}_{2n}} \cong & \Ind_{P^{+}}^{G_{2n}^{+}} (\pi_1 \widehat{\otimes} \cdots \widehat{\otimes} \pi_r |_{L^{+}}) \\
    \cong & (\Ind^{G^{+}_{2n}}_{P^{\circ}} \pi_1^{+} \widehat{\otimes} \cdots \widehat{\otimes} \pi_{r-1}^{+} \widehat{\otimes} \pi_r^{+})
    \oplus (\Ind^{G^{+}_{2n}}_{P^{\circ}} \pi_1^{+} \widehat{\otimes} \cdots \widehat{\otimes} \pi_{r-1}^{+} \widehat{\otimes} \pi_r^{-}).
    \end{aligned}
    $$
\end{proof}

For $\GL_2(\BR)$, as we have mentioned below Theorem \ref{Thm: shalika parabolic}, the Shalika period is closely related to the Whittaker period, which has been thoroughly studied in \cite{godement2018notes}. We have the following lemma.

\begin{lemt}\label{Lem: restrict GL2}
    Let $\pi := D_{k,\lambda}$, $k \in \BZ_{\geq 0},\ \lambda \in \BC$, then $\epsilon_{\pi} = \sgn a$.
\end{lemt}

\begin{proof}
    The case when $a>0$ has been calculated in \cite[Section 2.5]{godement2018notes} and \cite[Section 3.1]{Chen2019Deligne}. Then the case when $a<0$ can be obtained by twisting  
    $\begin{bmatrix}
    -1 & 0\\
    0 & 1
    \end{bmatrix}.$
\end{proof}

For $\GL_4(\BR)$, we have the following lemma, which is similar to Theorem \ref{Thm: shalika GL4}.

\begin{lemt}\label{Lem: restrict GL4}
    Let $P \subset \GL_4(\BR)$ be the standard parabolic subgroup corresponding to the partition $(2,2)$. Denote by $\eta : \BR^{\times} \to \BC^{\times},\ t \mapsto |t|^{z_0} (\sgn t)^{m_0}$ a character of $\BR^{\times}$. Assume $D_{k,\lambda}\ (k \in \BZ_{\geq 0},\ \lambda \in \BC)$ doesn't admit a $\eta$-twisted Shalika periods (thus so is $D_{k, z_0 - \lambda}$), and $\pi := D_{k,\lambda} \dot{\times} D_{k, z_0 - \lambda}$ is irreducible, then 
    $$D^{+}_{k,\lambda} \dot{\times} D^{-}_{k, z_0 - \lambda} =\Ind^{\GL^{+}_{4}(\BR)}_{P^{\circ}} D^{+}_{k,\lambda} \widehat{\otimes} D^{-}_{k, z_0 - \lambda}$$
    admits a $(\eta,\psi)$-twisted Shalika peroids. Moreover, we have
    $$
    \oH^{\CS}_{i}(S_4 ,D^{+}_{k,\lambda} \dot{\times} D^{-}_{k, z_0 - \lambda} \otimes \xi^{-1}_{\eta,\psi}) \cong \oH^{\CS}_{i}(\GL^{+}_2(\BR) ,D^{+}_{k,\lambda} \widehat{\otimes} D^{-}_{k, z_0 - \lambda} \otimes \eta_{\GL^{+}_2(\BR)}^{-1}),\quad i\in \BZ,
    $$
    and
    $$
    \Hom_{S_4}(D^{+}_{k,\lambda} \dot{\times} D^{-}_{k, z_0 - \lambda}, \xi_{\eta,\psi}) \cong \Hom_{\GL^{+}_2(\BR)}(D^{+}_{k,\lambda} \widehat{\otimes} D^{-}_{k, z_0 - \lambda}, \eta_{\GL^{+}_2(\BR)}).
    $$
    Thus $\epsilon_{\pi} = -1$.
\end{lemt}

\begin{proof}
    The proof is similar to the one of Theorem \ref{Thm: shalika GL4}, we only sketch it here for the convenience of the readers. Let 
    $$
    g' := 
    \begin{bmatrix}
    \begin{smallmatrix}
    1 & & &\\ 
    & 1& &\\
    & & 1 & \\
    & &  & -1
    \end{smallmatrix}
    \end{bmatrix}
    ,
    g_0 := 
    \begin{bmatrix}
    \begin{smallmatrix}
    1 & & &\\ 
    & -1& &\\
    & & 1 & \\
    & &  & -1
    \end{smallmatrix}
    \end{bmatrix}.
    $$
    Recall that in Theorem \ref{Thm: shalika GL4}, we have the following representative
    $\{\sigma_1 = (1,2,3,4), \sigma_2 = (1,3,2,4), \sigma_3 = (3,1,2,4), \sigma_4 = (3,4,1,2)\}.    $
    In this case, we have
    $$
    P^{\circ} \backslash G^{+}_{2n} / S_4 = \{ \sigma_1, g'\sigma_2, g_0g'\sigma_2, \sigma_3, g_0\sigma_3, \sigma_4\}.
    $$
    One can use the same argument as in Theorem \ref{Thm: shalika GL4} to show that 
    $$
    \oH^{\CS}_{i}(S_4 ,D^{+}_{k,\lambda} \dot{\times} D^{-}_{k, z_0 - \lambda} \otimes \xi^{-1}_{\eta,\psi}) \cong \oH^{\CS}_{i}(\GL^{+}_2(\BR) ,D^{+}_{k,\lambda} \widehat{\otimes} D^{-}_{k, z_0 - \lambda} \otimes \eta_{\GL^{+}_2(\BR)}^{-1}),\quad i\in \BZ,
    $$
    where the right term has dimension one for $i=0$.
    Thus 
    $$
    \dim \Hom_{S_4}(D^{+}_{k,\lambda} \dot{\times} D^{-}_{k, z_0 - \lambda}, \xi_{\eta,\psi}) = 1,
    $$
    and $\epsilon_{\pi} = -1$
\end{proof}

The following lemma is an analog of Theorem \ref{Thm: shalika parabolic} for $G^{+}_{2n}$.

\begin{lemt}\label{Lem: restrict parabolic ind}
     For two even positive integers $n_1$ and $n_2$, take two Casselman-Wallach representations $\pi_1$ and $\pi_2$ of $\GL^{+}_{n_1}(\BR)$ and $\GL^{+}_{n_2}(\BR)$, respectively. Then the normalized parabolic induction $\pi := \Ind^{\GL^{+}_{n_1+n_2}(\BR)}_{P^{\circ}_{n_1,n_2}} \pi_1 \widehat{\otimes} \pi_2$ has a non-zero $(\eta, \psi)$-Shalika periods if one of the following condition holds:
     \begin{itemize}
        \item [(1)]
        $n_1 \equiv 0$ or $n_2 \equiv 0 \mod 4$, and both $\pi_1$ and $\pi_2$ have $(\eta, \psi)$-Shalika periods.
        \item [(2)]
        $n_1 \equiv 2$ and $n_2 \equiv 2 \mod 4$, $\pi_1$ has $(\eta, \psi^{-1})$-Shalika period and $\pi_2$ has $(\eta, \psi)$-Shalika period.
    \end{itemize}
\end{lemt}

\begin{proof}
    Let $n_i = 2 m_i$,
    $
    \omega_s^{'} := 
    \begin{bmatrix}
    \begin{smallmatrix}
    1_{m_1} & & &\\ 
    & 0& 1_{m_2} &\\
    & 1_{m_1}& 0 & \\
    & &  & 1_{m_2}
    \end{smallmatrix}
    \end{bmatrix},
    $
    then $\det \omega_s^{'} = (-1)^{m_1m_2}$. Thus for the first case, $\omega_s^{'} \in G^{+}_{2n}$ and the proof is exactly the same as the original one in \cite[Theorem 2.1]{chen2020archimedean}. We only treat the remaining case and point out the difference.
    For the second case, take
    $
    \omega_s := 
    \begin{bmatrix}
    \begin{smallmatrix}
    -1_{m_1} & & &\\ 
    & 0& 1_{m_2} &\\
    & 1_{m_1}& 0 & \\
    & &  & 1_{m_2}
    \end{smallmatrix}
    \end{bmatrix}
    \in G_{2n}^{+}.
    $
    Then we have
    $$
    \omega_s^{-1} 
    \begin{bmatrix}
    \begin{smallmatrix}
    1_{m_1} &0 & X & Z\\ 
    & 1_{m_2}& 0 & Y\\
    & & 1_{m_1} & 0\\
    & &  & 1_{m_2}
    \end{smallmatrix}
    \end{bmatrix}
    \omega_s = 
    \begin{bmatrix}
    \begin{smallmatrix}
    1_{m_1} &-X & 0 & -Z\\ 
    & 1_{m_2}& 0 & 0\\
    & & 1_{m_1} & Y\\
    & &  & 1_{m_2}
    \end{smallmatrix}
    \end{bmatrix},
    $$
    and
    $$
    \omega_s^{-1} 
    \begin{bmatrix}
    \begin{smallmatrix}
    1_{m_1} &X & 0 & 0\\ 
    & 1_{m_2}& 0 & 0\\
    & & 1_{m_1} & Y\\
    & &  & 1_{m_2}
    \end{smallmatrix}
    \end{bmatrix}
    \omega_s = 
    \begin{bmatrix}
    \begin{smallmatrix}
    1_{m_1} &0 & -X & 0\\ 
    & 1_{m_2}& 0 & Y\\
    & & 1_{m_1} & 0\\
    & &  & 1_{m_2}
    \end{smallmatrix}
    \end{bmatrix}.
    $$
    For $f \in \Ind^{\GL^{+}_{n_1+n_2}(\BR)}_{P^{\circ}_{n_1,n_2}} \pi_1 \widehat{\otimes} \pi_2$, we define a function on $G_{2n}^+$ by
    $
    \Phi(g;f) := <\mu_1 \otimes \mu_2, f(\omega_s^{-1}g)>
    $,
    where $\mu_i$ is the non-zero twisted Shalika functional on $\pi_i$. Then we obtain
    $$
    \begin{aligned}
    \Phi \left(
    \begin{bmatrix}
    \begin{smallmatrix}
    1_{m_1} &0 & X & Z\\ 
    & 1_{m_2}& 0 & Y\\
    & & 1_{m_1} & 0\\
    & &  & 1_{m_2}
    \end{smallmatrix}
    \end{bmatrix} 
    g ; f
    \right)
    = & \psi^{-1}(-\Tr(X)) \cdot \psi(\Tr(Y)) \cdot \Phi(g;f)\\
    = & \psi(\Tr(X)+\Tr(Y)) \cdot \Phi(g;f)
    \end{aligned},
    $$
    and
    $$
    \Phi \left(
    \begin{bmatrix}
    \begin{smallmatrix}
    1_{m_1} &X & 0 & 0\\ 
    & 1_{m_2}& 0 & 0\\
    & & 1_{m_1} & Y\\
    & &  & 1_{m_2}
    \end{smallmatrix}
    \end{bmatrix} 
    g ; f
    \right)
    = \Phi(g;f).
    $$
    The subsequent construction of the twisted Shalika functional on $\pi$ follows the same process as in \cite[Theorem 2.1]{chen2020archimedean}.
\end{proof}

We are now prepared to state the main theorem in this section.

\begin{thm}\label{Thm: restrict Shalika support}
    Let $\pi$ be an irreducible generic representation of $G_{2n}$. Assume $\pi$ admits a non-zero $(\eta,\psi_a)$-twisted Shalika period and $\pi|_{G^{+}_{2n}}$ is reducible, then $\pi$ has form
    $$
    D_{k_1, \lambda_1} \dot{\times} \cdots \dot{\times} D_{k_n, \lambda_n},\ k_i \in \BZ_{\geq 0},\ \lambda_i \in \BC.
    $$
    Let
    $
    p := \#\{1 \leq i \leq n |D_{a_i, \lambda_i} \text{ has } (\eta,\psi_a)\text{-twisted Shalika period} \}
    $
    and $q := \frac{n-p}{2}$, which are both integers since $\phi_{\pi}$ is of $\eta$-symplectic type. Then
        $$
        \epsilon_{\pi} = (\sgn\ a)^p \cdot (-1)^{\frac{p(p-1)}{2}+q}.
        $$
\end{thm}

\begin{proof}
    Applying Lemma \ref{Lem: restrict GL2}, \ref{Lem: restrict GL4}, \ref{Lem: restrict parabolic ind}, the theorem is derived through an induction argument.
\end{proof}

\appendix

\section{\kun formula and spectral sequence for nilpotent normal subgroups}\label{Sec: Appendix A}

\subsection{\kun formula}\label{Subsec: kunneth}
In this subsection, we prove a \kun formula for Schwartz homology.
Since $\CS\mathrm{mod}_{G}$ is not an Abelian category, we first prove the following two basic lemma about homological algebra in the topological setting.

\begin{lemt}[Topological snake lemma]\label{Lem: snake lem}
Consider a commutative diagram in the category of (not necessarily Hausdorff) topological vector spaces
\[
\begin{tikzcd}
   & V_1 \arrow[r, "\alpha_1"] \arrow[d, "\phi_1"] & V_2 \arrow[r, "\alpha_2"] \arrow[d, "\phi_2"] & V_3 \arrow[r] \arrow[d, "\phi_3"] & 0 \\
  0 \arrow[r] & W_1 \arrow[r, "\beta_1"]  & W_2 \arrow[r, "\beta_2"]  & W_3  &  
\end{tikzcd}
\]
where the rows are weakly exact sequences. If $\alpha_2$ is an open map, $\beta_1 : W_1 \xrightarrow{\sim} \Ima \beta_1$ is a topological isomorphism, then we have a weakly short exact sequence
\[
\Ker(\phi_1) \xrightarrow{} \Ker(\phi_2) \xrightarrow{} \Ker(\phi_3) \xrightarrow{\delta} \operatorname{coker}(\phi_1) \xrightarrow{} \operatorname{coker}(\phi_2) \xrightarrow{} \operatorname{coker}(\phi_3)
\]
where $\Ker(\phi_i)$ are equipped with subspace topology of $V_i$, $\operatorname{coker}(\phi_i)$ are equipped with quotient topology of $W_i$.
\end{lemt}

\begin{proof}
    The conditions of this lemma are slightly different from those of \cite[Proposition 4]{schochet1999topological}, but the proof is identical. We refer the readers to \cite{schochet1999topological} for details.
\end{proof}

\begin{lemt}\label{Lem: short long seq}
    Let 
    \[
    0 \xrightarrow{} A_.  \xrightarrow{\phi_.} B_.  \xrightarrow{\eta_.} C_. \xrightarrow{} 0
    \]
    be a short exact sequence of chain complexes of \Fre spaces, where each $\phi_i$ is a strict map. Then we have a weakly exact sequence
    \[
    \cdots \xrightarrow{} \oH_{i+1}(C_.) \xrightarrow{} \oH_{i}(A_.)  \xrightarrow{} \oH_{i}(B_.)  \xrightarrow{} \oH_{i}(C_.) \xrightarrow{} \oH_{i-1}(A_.) \xrightarrow{} \cdots
    \]
\end{lemt}
\begin{proof}
    This follows the standard argument based on Lemma \ref{Lem: snake lem}. We omit the details.
\end{proof}

We also need the following lemma since we work in the topological setting.

\begin{lemt}[{\cite[Lemma 8.6]{CHEN2021108817}}]\label{Lem: homology open map}
    Let $\phi : \CC_{\bullet}\rightarrow\CC^{\prime}_{\bullet}$ be a morphism of chain complexes of \Fre spaces. Let $i\in \BZ$. If the induced morphism
    \[
    \phi: \oH_{i}(\CC_{\bullet})\rightarrow\oH_{i}(\CC^{\prime}_{\bullet})
    \]
    is surjective, then it must be an open map.
\end{lemt}

We first define the completed tensor product of two bounded below chain complexes of topological vector spaces. For simplicity, we shall assume the complex is started from $0$. 
\begin{dfnt}
    Let $X_., Y_.$ be two bounded below chain complexes of topological vector spaces, both starting from 0. We define the completed tensor product of $X_.$ and $ Y_.$ by
    $$
    (X_. \widehat{\otimes} Y_. )_m = \bigoplus_{p+q=m} X_p \widehat{\otimes} Y_q, \quad \forall\ m \in \BZ
    $$
    with the direct sum topology. The chain maps are the same as the algebraic tensor product of chain complexes.
\end{dfnt}

\begin{thm}[\kun formula for chain complex]\label{Thm: kunneth chain}
    Let $X_., Y_.$ be two bounded below chain complexes of NF-spaces, both starting from 0. Assume $\forall\ k \in \BZ$, $\oH_k(X_.),\ \oH_k(Y_.)$ are NF-spaces. Then we have
    $$
    \oH_m(X_. \widehat{\otimes} Y_. ) \cong \bigoplus_{p+q=m} \oH_p(X_.) \widehat{\otimes} \oH_q(Y_.), \quad \forall\ m \in \BZ,
    $$
    as topological vector spaces. In particular, $\oH_m(X_. \widehat{\otimes} Y_. )$ are NF-spaces. 
\end{thm}

\begin{proof}
    The proof is similar to the classical case, except for certain topological issues (where we use the nuclear \fre condition). We sketch the proof for the convenience of the readers.
    \par
    We first consider the case that $Y_i = 0$ for $i \neq p$, for some $p \in \BZ$. For $m \in \BZ$, according to Lemma \ref{Lem: NF tensor exact}, we have short exact sequences of NF-spaces
    \[
     0 \xrightarrow{} \Ker\ d_{m} \widehat{\otimes} Y_p \xrightarrow{} X_{m} \widehat{\otimes} Y_p  \xrightarrow{d_{m}\widehat{\otimes} 1_{Y_p}} \Im\ d_{m} \widehat{\otimes} Y_p \xrightarrow{} 0,
    \]
    \[
     0 \xrightarrow{} \Im\ d_{m+1} \widehat{\otimes} Y_p  \xrightarrow{} \Ker\ d_{m} \widehat{\otimes} Y_p \xrightarrow{} \oH_{m}(X_.) \widehat{\otimes} Y_p \xrightarrow{} 0.
    \]
    Then we know that $(\Im\ d_{m+1}) \widehat{\otimes} Y_p = \Im\ (d_{m+1} \widehat{\otimes} 1_{Y_p})$, $(\Ker\ d_{m}) \widehat{\otimes} Y_p = \Ker\ (d_{m} \widehat{\otimes} 1_{Y_p})$. 
    Thus $\oH_{m+p}(X_. \widehat{\otimes} Y_. ) = \oH_{m}(X_.) \widehat{\otimes} Y_p = \oH_{m}(X_.) \widehat{\otimes} \oH_p(Y_.)$.
    \par
    Secondly, we consider the case that the differential map $d^Y_{.}$ in $Y_.$ are all $0$. Note that $(X_. \widehat{\otimes} Y_. )_{m} = \bigoplus_{p+q=m} X_p \widehat{\otimes} Y_q$. Since $d^Y_{.} = 0$, the differential map $(X_. \widehat{\otimes} Y_. )_{m+1} \to (X_. \widehat{\otimes} Y_. )_{m}$ is just direct sum of family of maps $d^X_{p+1} \widehat{\otimes} 1_{Y_q} : X_{p+1} \widehat{\otimes} Y_q \to X_p \widehat{\otimes} Y_q$. Since all the occurring spaces are NF-spaces, we have $(\Im\ d^X_{p+1}) \widehat{\otimes} Y_q = \Im\ (d^X_{p+1} \widehat{\otimes} 1_{Y_q})$, $(\Ker\ d^X_{p}) \widehat{\otimes} Y_q = \Ker\ (d^X_{p} \widehat{\otimes} 1_{Y_q})$. Thus we obtain 
    $$
    \oH_{m}(X_. \widehat{\otimes} Y_. ) 
    = \frac{\bigoplus_{p+q=m} (\Ker\ d^X_{p}) \widehat{\otimes} Y_q}{\bigoplus_{p+q=m} (\Im\ d^X_{p+1}) \widehat{\otimes} Y_q}
    = \bigoplus_{p+q=m} \oH_p(X_.) \widehat{\otimes} Y_q
    = \bigoplus_{p+q=m} \oH_p(X_.) \widehat{\otimes} \oH_q(Y_.).
    $$
    \par
    For the general case, denote $Z(X)_m := \Ker\ d^X_m$, $B(X)_m := \Im\ d^X_{m+1}$. Let $Z(X)_.$ and $B(X)_.$ be the corresponding chain complexes with zero differential maps. Define the chain complex $\Sigma B(X)$ by setting $(\Sigma B(X))_{n+1}:= B(X)_n $. 
    We have a short exact sequence of chain complexes
    \[
     0 \xrightarrow{} Z(X)_. \xrightarrow{} X_.  \xrightarrow{d^X_.} \Sigma B(X)_. \xrightarrow{} 0.
    \]
    According to the condition, all these spaces are NF-spaces, thus we have
    \[
     0 \xrightarrow{} Z(X)_. \widehat{\otimes} Y_. \xrightarrow{} X_. \widehat{\otimes} Y_. \xrightarrow{d^X_.} \Sigma B(X)_. \widehat{\otimes} Y_. \xrightarrow{} 0,
    \]
    with the first maps being strict. Following Lemma \ref{Lem: short long seq}, we obtain
    \begin{multline*}
    \cdots \to \oH_{m+1}(\Sigma B(X)_. \widehat{\otimes} Y_.) \xrightarrow{} \oH_{m}(Z(X)_. \widehat{\otimes} Y_.)  \xrightarrow{} \oH_{m}(X_. \widehat{\otimes} Y_.) \\ 
    \xrightarrow{} \oH_{m}(\Sigma B(X)_. \widehat{\otimes} Y_.) \xrightarrow{} \oH_{m-1}(Z(X)_. \widehat{\otimes} Y_.) \to \cdots.
    \end{multline*}
    The first two terms, as well as the last two, can be computed using the second case, then we get the exact sequence
    \begin{multline*}
    \cdots \to \bigoplus_{p+q=m} B(X)_p \widehat{\otimes} \oH_{q}(Y_.) \xrightarrow{} \bigoplus_{p+q=m} Z(X)_p \widehat{\otimes} \oH_{q}(Y_.)  \xrightarrow{} \oH_{m}(X_. \widehat{\otimes} Y_.)  \\ 
    \xrightarrow{} \bigoplus_{p+q=m-1} B(X)_p \widehat{\otimes} \oH_{q}(Y_.) \xrightarrow{} \bigoplus_{p+q=m-1} Z(X)_p \widehat{\otimes} \oH_{q}(Y_.) \to \cdots,
    \end{multline*}
    where the connecting homomorphism is induced by the inclusion 
    $$ B(X)_p \widehat{\otimes} \oH_{q}(Y_.) \hookrightarrow Z(X)_p \widehat{\otimes} \oH_{q}(Y_.).$$
    Thus we have the short exact sequence
    \[
    0 \xrightarrow{} \bigoplus_{p+q=m} B(X)_p \widehat{\otimes} \oH_{q}(Y_.) \xrightarrow{} \bigoplus_{p+q=m} Z(X)_p \widehat{\otimes} \oH_{q}(Y_.)  \xrightarrow{} \oH_{m}(X_. \widehat{\otimes} Y_.)  \xrightarrow{} 0.
    \]
    Following Lemma \ref{Lem: homology open map}, the surjective map above is an open map. Also, we know
    \[
    0 \xrightarrow{} B(X)_p \widehat{\otimes} \oH_{q}(Y_.) \xrightarrow{}  Z(X)_p \widehat{\otimes} \oH_{q}(Y_.)  \xrightarrow{} \oH_{p}(X_.) \widehat{\otimes} \oH_{q}(Y_.)  \xrightarrow{} 0,
    \]
    thus we get 
    $$
    \oH_{m}(X_. \widehat{\otimes} Y_. ) \cong \bigoplus_{p+q=l} \oH_p(X_.) \widehat{\otimes} \oH_q(Y_.)
    $$
    as topological vector spaces.
\end{proof}

Now we focus on the Schwartz homology.

\begin{lemt}\label{Lem: kunneth trivial}
    Let $G$ be an almost linear Nash group.
    Assume that $V,\ W \in \CS\mathrm{mod}_{G}$ are both NF-spaces. If $\forall\ m \in \BZ$, $\oH^{\CS}_{m}(G, V)$ are NF-spaces, then
    $$
    \oH^{\CS}_m(G, V \widehat{\otimes} W) = \oH^{\CS}_m(G, V) \widehat{\otimes} W, \quad \forall m \in \BZ,
    $$
    as topological vector spaces.
\end{lemt}

\begin{proof}
    Take $P_. \twoheadrightarrow V$ to be a strong relative projective resolution of $V$, with $P_.$ all NF-spaces, 
    then $P_. \widehat{\otimes} W \twoheadrightarrow V \widehat{\otimes} W$ is a strong relative projective resolution of $V \widehat{\otimes} W$. 
    \par
    We first compare $\sum(g-1)(P_k \widehat{\otimes} W)$ and $(\sum(g-1)P_k) \widehat{\otimes} W$, which are both closed subspace in $P_k \widehat{\otimes} W$. Note that $\sum(g-1)(P_k \otimes W) = (\sum(g-1)P_k) \otimes W$, $\sum(g-1)(P_k \otimes W)$ is dense in $\sum(g-1)(P_k \widehat{\otimes} W)$, $(\sum(g-1)P_k) \otimes W$ is dense in $(\sum(g-1)P_k) \widehat{\otimes} W$. Thus we have $\sum(g-1)(P_k \widehat{\otimes} W) = (\sum(g-1)P_k) \widehat{\otimes} W$, and $(P_k \widehat{\otimes} W)_G = (P_k)_G \widehat{\otimes} W$.
    \par
    Consider the following sequence of NF-spaces
    \[
     \cdots \xrightarrow{} (P_{m+1})_G \xrightarrow{d_{m+1}} (P_m)_G  \xrightarrow{d_{m}} (P_{m-1})_G \xrightarrow{} \cdots.
    \]
    According to Lemma \ref{Lem: NF tensor exact}, we have $(\Im\ d_{m+1}) \widehat{\otimes} W = \Im\ (d_{m+1} \widehat{\otimes} 1_{W})$, $(\Ker\ d_{m}) \widehat{\otimes} W = \Ker\ (d_{m} \widehat{\otimes} 1_{W})$. Note that
    \[
     0 \xrightarrow{} (\Im\ d_{m+1}) \widehat{\otimes} W  \xrightarrow{} (\Ker\ d_m) \widehat{\otimes} W \xrightarrow{} \oH^{\CS}_{m}(G, V) \widehat{\otimes} W \xrightarrow{} 0,
    \]
    thus 
    $
    \oH^{\CS}_m(G, V \widehat{\otimes} W) = \oH^{\CS}_m(G, V) \widehat{\otimes} W, \quad \forall m \in \BZ.$
\end{proof}

\begin{thm}[\kun formula for Schwartz homology]\label{Thm: kunneth Schwartz homology}
    Let $G_1$ and $G_2$ be two almost linear Nash groups.
    Assume $V_i \in \CS\mathrm{mod}_{G_i}$, $i = 1,2$, are both NF-spaces. If $\forall\ j \in \BZ$, $\oH^{\CS}_{j}(G_i, V_i)$ are NF-spaces, then 
    $$
    \oH^{\CS}_m(G_1 \times G_2, V_1 \widehat{\otimes} V_2) \cong \bigoplus_{p+q=m} \oH^{\CS}_p(G_1, V_1) \widehat{\otimes} \oH^{\CS}_q(G_2, V_2), \quad \forall\ m \in \BZ,
    $$
    as topological vector spaces. In particular, $\oH^{\CS}_m(G_1 \times G_2, V_1 \widehat{\otimes} V_2)$ are NF-spaces. 
\end{thm}

\begin{proof}
    Using Theorem \ref{Thm: kunneth chain}, it suffices to prove that the Koszul complex for computing the Schwartz homology of $V_1 \widehat{\otimes} V_2$ in $\smod_{G_1 \times G_2}$ is just the completed tensor product of the Koszul complexes for $V_i$ in $\smod_{G_i}\ (i = 1,2)$. Denote by $K_i\ (i=1,2)$ the maximal compact subgroup of $G_i\ (i=1,2)$.
    \par
    Since $V_i\ (i=1,2)$ are both NF-spaces and $K_i\ (i=1,2)$ are both compact, we have
    \[
     (\wedge^{n}(\g_1 \oplus \g_2 / \k_1 \oplus \k_2 )\otimes V_1 \widehat{\otimes} V_2)_{K_1 \times K_2} =
     \bigoplus_{p+q=n }((\wedge^{p}(\g_1 / \k_1)\otimes V_1) \widehat{\otimes} (\wedge^{q}(\g_2 /\k_2 )\otimes V_2))_{K_1 \times K_2}.
    \]
    Thus we only need to consider $((\wedge^{p}(\g_1 / \k_1)\otimes V_1) \widehat{\otimes} (\wedge^{q}(\g_2 /\k_2 )\otimes V_2))_{K_1 \times K_2}$. Note that $K_i\ (i=1,2)$ are both compact, thus $\oH^{\CS}_{0}(K_i, \wedge^{p}(\g_i / \k_i)\otimes V_i)\ (i=1,2)$ are both NF-spaces according to Proposition \ref{Prop: projective hausdorff}. Following Lemma \ref{Lem: kunneth trivial}, we have
    $$
    \oH^{\CS}_{0}(K_1 \times K_2, (\wedge^{p}(\g_1 / \k_1)\otimes V_1) \widehat{\otimes} (\wedge^{q}(\g_2 /\k_2 )\otimes V_2))
    =\oH^{\CS}_{0}(K_1 , \wedge^{p}(\g_1 / \k_1)\otimes V_1) \widehat{\otimes} \oH^{\CS}_{0}(K_2 , \wedge^{q}(\g_2 /\k_2 )\otimes V_2).
    $$
\end{proof}

\subsection{Hochschild-Serre spectral sequence}\label{Subsec: spectral sequence}
In this subsection, we aim to prove a Hochschild-Serre spectral sequence for nilpotent normal subgroups.
Firstly, we recall the following spectral sequences in the category $\CC_{\g_{\BC}, K}$ of $(\g_{\BC}, K)$-modules.
\begin{thm}[{\cite[Theorem 6.5]{borel2000continuous}}, {\cite[Corollary 3.6]{knapp2016cohomological}}]\label{Thm: spec seq for lie alg}
    Let $\h$ be an ideal of $\g$ stable under $K$, $\l = \k \cap \h$, $L$ a closed normal subgroup of $K$ with Lie algebra $\l$, and $W \in \CC_{\g_{\BC},K}$. Then there exists convergent first quadrant spectral sequences:
    $$
    E^2_{p,q} = \oH_p(\g_{\BC}/\h_{\BC},K/L; H_q(\h_{\BC},L ; W)) \implies H_{p+q}(\g_{\BC},K ;W)
    $$
\end{thm}

Since we work in $\smod_G$, it is desirable to obtain a similar spectral sequence applicable to this category. The following lemma is crucial.

\begin{lemt}\label{Lem: nilp homology comparison}
    Let $H = L \ltimes N$ be an almost linear Nash group with nilpotent normal subgroup $N$, denote by $M$ the maximal compact subgroup of $L$ (and hence also the maximal compact subgroup in $H$). Consider $V \in \CS\mathrm{mod}_H$, for a fixed $j \in \BZ$, if $\oH^{\CS}_{j}(N, V) \in \CS\mathrm{mod}_L$, then $\oH_{j}(\n_{\BC}, V^{M\text{-fin}}) = \oH^{\CS}_{j}(N, V)^{M\text{-fin}}$
\end{lemt}

\begin{proof}
    Consider the Koszul complex for $\oH_{j}(\n_{\BC}, V)$:
    \[
     \cdots \xrightarrow{} \wedge^{j+1}\n_{\BC} \otimes V  \xrightarrow{d_{j+1}} \wedge^{j}\n_{\BC} \otimes V \xrightarrow{d_j} \wedge^{j-1}\n_{\BC} \otimes V \xrightarrow{} \cdots.
    \]
    Since $\oH_{j}(\n_{\BC}, V) = \oH^{\CS}_{j}(N, V) \in \CS\mathrm{mod}_L$, we have the following short exact sequence in $\CS\mathrm{mod}_L$
    \[
     0 \xrightarrow{} \Im\ d_{j+1}  \xrightarrow{} \Ker\ d_j  \xrightarrow{} \oH_{j}(\n_{\BC}, V) \xrightarrow{} 0.
    \]
    Note that taking $M$-finite vectors $-^{M\text{-fin}}$ is an exact functor from $\CS\mathrm{mod}_L$ to $(\l_{\BC},M)$-modules, thus we have
    \[
     0 \xrightarrow{} (\Im\ d_{j+1})^{M\text{-fin}}  \xrightarrow{} (\Ker\ d_j)^{M\text{-fin}}  \xrightarrow{} \oH_{j}(\n_{\BC}, V)^{M\text{-fin}} \xrightarrow{} 0.
    \]
    We also have the following Koszul complex for $\oH_{j}(\n_{\BC}, V^{M\text{-fin}})$:
    \[
     \cdots \xrightarrow{} \wedge^{j+1}\n_{\BC} \otimes V^{M\text{-fin}}  \xrightarrow{d_{j+1}|_{M\text{-fin}}} \wedge^{j}\n_{\BC} \otimes V^{M\text{-fin}} \xrightarrow{d_j|_{M\text{-fin}}} \wedge^{j-1}\n_{\BC} \otimes V^{M\text{-fin}} \xrightarrow{} \cdots.
    \]
    Then we obtain
    \[
     0 \xrightarrow{} \Im\ (d_{j+1}|_{M\text{-fin}})  \xrightarrow{} \Ker\ (d_j|_{M\text{-fin}})  \xrightarrow{} \oH_{j}(\n_{\BC}, V^{M\text{-fin}}) \xrightarrow{} 0.
    \]
    Thus it suffices to prove 
    $$\Im (d_{j+1}|_{M\text{-fin}}) = (\Im d_{j+1})^{M\text{-fin}},\ \Ker (d_j|_{M\text{-fin}}) = (\Ker d_j)^{M\text{-fin}}.$$ 
    The first equality follows from the following short exact sequence
    \[
     0 \xrightarrow{} (\Ker\ d_{j+1})^{M\text{-fin}} \xrightarrow{} (\wedge^{j+1}\n_{\BC} \otimes V)^{M\text{-fin}}  \xrightarrow{ d_{j+1}|_{M\text{-fin}}} (\Im\ d_{j+1})^{M\text{-fin}} \xrightarrow{} 0,   \]
     and the second follows from 
    $\Ker\ (d_j|_{M\text{-fin}}) = (\Ker\ d_j) \cap (\wedge^{j}\n_{\BC} \otimes V^{M\text{-fin}}) = (\Ker\ d_j)^{M\text{-fin}}.$
\end{proof}

\begin{rremark}
    This Lemma is unrelated to Casselman's conjectured comparison theorem, since they are under different conditions.
\end{rremark}

\begin{cort}[Hochschild-Serre spectral sequence for nilpotent normal subgroups]\label{Cor: spec seq nilp}
    The notation is as shown in Lemma \ref{Lem: nilp homology comparison}.
    Consider $V \in \CS\mathrm{mod}_H$, if $\forall\ j \in \BZ$, $\oH^{\CS}_{j}(N, V) \in \CS\mathrm{mod}_L$, then we have 
    $$\oH_{i}(\l_{\BC}, M; \oH_{j}(\n_{\BC} ; V^{M\text{-fin}})) = \oH^{\CS}_{i}(L, \oH^{\CS}_{j}(N, V)),\ \forall\ i,\ j \in \BZ.$$
    Moreover, there exist convergent first quadrant spectral sequences:
    $$
    E^2_{p,q} = \oH^{\CS}_{p}(L, \oH^{\CS}_{q}(N, V)) \implies \oH^{\CS}_{p+q}(H ,V).
    $$
\end{cort}

\begin{proof}
    Following Lemma \ref{Lem: nilp homology comparison} and Proposition \ref{Prop: homology comparison}, 
    $$\oH_{i}(\l_{\BC}, M; \oH_{j}(\n_{\BC} ; V^{M\text{-fin}})) = \oH_{i}(\l_{\BC}, M; \oH^{\CS}_{j}(N, V)^{M\text{-fin}}) = \oH^{\CS}_{i}(L, \oH^{\CS}_{j}(N, V)).$$
    As to the second assertion, according to Theorem \ref{Thm: spec seq for lie alg}, we have 
    $$
    E^2_{p,q} = \oH_p(\l_{\BC},M; H_q(\n_{\BC} ; V^{M\text{-fin}})) \implies H_{p+q}(\h_{\BC},M ;V^{M\text{-fin}})
    $$
    Then the corollary follows from the first assertion and Proposition \ref{Prop: homology comparison}.
\end{proof}

\bibliographystyle{alpha}
\bibliography{ref}

\newcommand{\etalchar}[1]{$^{#1}$}
\begin{thebibliography}{ALM{\etalchar{+}}24}

\bibitem[AGS15]{aizenbud2015twisted}
Avraham Aizenbud, Dmitry Gourevitch, and Siddhartha Sahi.
\newblock Twisted homology for the mirabolic nilradical.
\newblock {\em Israel Journal of Mathematics}, 206:39--88, 2015.

\bibitem[ALM{\etalchar{+}}24]{anandavardhanan2024sign}
U.~K. Anandavardhanan, Hengfei Lu, Nadir Matringe, Vincent Sécherre, and Chang Yang.
\newblock The sign of linear periods, 2024.

\bibitem[BK14]{Bernstein2014Smooth}
Joseph Bernstein and Bernhard Kr\"otz.
\newblock Smooth {F}r\'echet globalizations of {H}arish-{C}handra modules.
\newblock {\em Israel J. Math.}, 199(1):45--111, 2014.

\bibitem[BW80]{borel2000continuous}
Armand Borel and Nolan~R. Wallach.
\newblock {\em Continuous cohomology, discrete subgroups, and representations of reductive groups}, volume No. 94 of {\em Annals of Mathematics Studies}.
\newblock Princeton University Press, Princeton, NJ; University of Tokyo Press, Tokyo, 1980.

\bibitem[CC19]{Chen2019Deligne}
Shih-Yu Chen and Yao Cheng.
\newblock On {D}eligne's conjecture for certain automorphic {$L$}-functions for {$\rm GL(3)\times GL(2)$} and {$\rm GL(4)$}.
\newblock {\em Doc. Math.}, 24:2241--2297, 2019.

\bibitem[CHM00]{Casselman2000Bruhat}
William Casselman, Henryk Hecht, and Dragan Milicic.
\newblock Bruhat filtrations and {W}hittaker vectors for real groups.
\newblock In {\em The mathematical legacy of {H}arish-{C}handra ({B}altimore, {MD}, 1998)}, volume~68 of {\em Proc. Sympos. Pure Math.}, pages 151--190. Amer. Math. Soc., Providence, RI, 2000.

\bibitem[CJLT20]{chen2020archimedean}
Cheng Chen, Dihua Jiang, Bingchen Lin, and Fangyang Tian.
\newblock Archimedean non-vanishing, cohomological test vectors, and standard l-functions of $gl_{2n}$: real case.
\newblock {\em Mathematische Zeitschrift}, 296(1):479--509, 2020.

\bibitem[CO78]{casselman1978restriction}
William Casselman and M~Scott Osborne.
\newblock The restriction of admissible representations to n.
\newblock {\em Mathematische Annalen}, 233:193--198, 1978.

\bibitem[CS20]{chen2020uniqueness}
Fulin Chen and Binyong Sun.
\newblock Uniqueness of twisted linear periods and twisted shalika periods.
\newblock {\em Science China Mathematics}, 63:1--22, 2020.

\bibitem[CS21]{CHEN2021108817}
Yangyang Chen and Binyong Sun.
\newblock Schwartz homologies of representations of almost linear nash groups.
\newblock {\em Journal of Functional Analysis}, 280(7):108817, 2021.

\bibitem[FSX18]{fang2018godement}
Yingjue Fang, Binyong Sun, and Huajian Xue.
\newblock Godement--jacquet l-functions and full theta lifts.
\newblock {\em Mathematische Zeitschrift}, 289:593--604, 2018.

\bibitem[Gan19]{Gan2019}
Wee~Teck Gan.
\newblock Periods and theta correspondence.
\newblock In {\em Representations of reductive groups}, volume 101 of {\em Proc. Sympos. Pure Math.}, pages 113--132. Amer. Math. Soc., Providence, RI, 2019.

\bibitem[God18]{godement2018notes}
Roger Godement.
\newblock {\em Notes on {J}acquet-{L}anglands' theory}, volume~8 of {\em CTM. Classical Topics in Mathematics}.
\newblock Higher Education Press, Beijing, 2018.
\newblock With commentaries by Robert Langlands and Herve Jacquet.

\bibitem[HT98]{hecht1998remark}
Henryk Hecht and Joseph~L Taylor.
\newblock A remark on casselman’s comparison theorem.
\newblock In {\em Geometry and Representation Theory of Real and p-adic groups}, pages 139--146. Springer, 1998.

\bibitem[Jac09]{Jacquet2009RankinSelberg}
Herv\'e Jacquet.
\newblock Archimedean {R}ankin-{S}elberg integrals.
\newblock In {\em Automorphic forms and {$L$}-functions {II}. {L}ocal aspects}, volume 489 of {\em Contemp. Math.}, pages 57--172. Amer. Math. Soc., Providence, RI, 2009.

\bibitem[JLST24]{Jiang2024twisted}
Dihua Jiang, Dongwen Liu, Binyong Sun, and Fangyang Tian.
\newblock On twisted shalika functionals on the principal series representations of gl(2n).
\newblock preprint, 2024.

\bibitem[JST24]{jiang2024period}
Dihua Jiang, Binyong Sun, and Fangyang Tian.
\newblock Period relations for standard $l$-functions of symplectic type, 2024.

\bibitem[Kna94]{Knapp1994}
A.~W. Knapp.
\newblock Local {L}anglands correspondence: the {A}rchimedean case.
\newblock In {\em Motives ({S}eattle, {WA}, 1991)}, volume 55, Part 2 of {\em Proc. Sympos. Pure Math.}, pages 393--410. Amer. Math. Soc., Providence, RI, 1994.

\bibitem[KVJ16]{knapp2016cohomological}
Anthony~W Knapp and David~A Vogan~Jr.
\newblock {\em Cohomological Induction and Unitary Representations (PMS-45), Volume 45}, volume~28.
\newblock Princeton University Press, 2016.

\bibitem[LLY21]{li2021proof}
Ning Li, Gang Liu, and Jun Yu.
\newblock A proof of casselman’s comparison theorem.
\newblock {\em Representation Theory of the American Mathematical Society}, 25(35):994--1020, 2021.

\bibitem[LT20]{lin2020archimedean}
Bingchen Lin and Fangyang Tian.
\newblock Archimedean non-vanishing, cohomological test vectors, and standard l-functions of $gl_{2n}$: Complex case.
\newblock {\em Advances in Mathematics}, 369:107189, 2020.

\bibitem[Mat17]{matringe2017shalika}
Nadir Matringe.
\newblock Shalika periods and parabolic induction for gl (n) over a non-archimedean local field.
\newblock {\em Bulletin of the London Mathematical Society}, 49(3):417--427, 2017.

\bibitem[Moe97]{Moeglin1997}
C.~Moeglin.
\newblock Representations of {${\rm GL}(n)$} over the real field.
\newblock In {\em Representation theory and automorphic forms ({E}dinburgh, 1996)}, volume~61 of {\em Proc. Sympos. Pure Math.}, pages 157--166. Amer. Math. Soc., Providence, RI, 1997.

\bibitem[Sch99]{schochet1999topological}
Claude Schochet.
\newblock The topological snake lemma and corona algebras.
\newblock {\em New York Journal of Mathematics}, 5:131--137, 1999.

\bibitem[ST23]{suzuki2023epsilon}
Miyu Suzuki and Hiroyoshi Tamori.
\newblock Epsilon dichotomy for linear models: the archimedean case.
\newblock {\em International Mathematics Research Notices}, 2023(20):17853--17891, 2023.

\bibitem[Vog78]{Vogan1978gelfand}
David~A. Vogan, Jr.
\newblock Gelfand-{K}irillov dimension for {H}arish-{C}handra modules.
\newblock {\em Invent. Math.}, 48(1):75--98, 1978.

\bibitem[Wal92]{Wallach1992Real}
Nolan~R. Wallach.
\newblock {\em Real reductive groups. {II}}, volume 132-II of {\em Pure and Applied Mathematics}.
\newblock Academic Press, Inc., Boston, MA, 1992.

\bibitem[Xue20]{Xue20Bessel}
Hang Xue.
\newblock Bessel models for unitary groups and schwartz homology.
\newblock preprint, 2020.

\end{thebibliography}

\end{document}